\documentclass[11pt,a4paper]{article}

\usepackage{color}
\definecolor{red}{rgb}{1,0,0}

\definecolor{gre}{rgb}{0,0.5,0}

\definecolor{blu}{rgb}{0,0,1}
\definecolor{pur}{rgb}{0.6,0.16,0.9}

\usepackage[utf8]{inputenc}
\usepackage{amssymb}
\usepackage{amsmath}
\usepackage{mathtools}
\usepackage{amsfonts}
\usepackage{amsthm}
\usepackage{bbm}
\usepackage{bm}
\usepackage{hyperref}
\usepackage{physics}

\newcommand{\transpose}[1]{\ensuremath{\prescript{t}{}{#1}}}
\newcommand{\e}[0]{\ensuremath{\mathrm{e}}}
\renewcommand{\i}[0]{\ensuremath{\mathrm{i}}}

\newtheorem{theorem}{Theorem}[section]
\newtheorem{corollary}[theorem]{Corollary}
\newtheorem{lemma}[theorem]{Lemma}
\newtheorem{proposition}[theorem]{Proposition}

\theoremstyle{definition}

\numberwithin{equation}{section}

\usepackage[dvips, left=2.4cm, right=2.4cm , top=2.4cm, bottom=3.4cm]{geometry}

\begin{document}
\title{Bounds for theta sums in higher rank I\footnote{Research supported by EPSRC grant EP/S024948/1}}
\date{School of Mathematics, University of Bristol, Bristol BS8 1UG, U.K.\\[10pt] 22 March 2022}
\author{Jens Marklof and Matthew Welsh}
\maketitle

\begin{abstract}
Theta sums are finite exponential sums with a quadratic form  in the oscillatory phase. This paper establishes new upper bounds for theta sums in the case of smooth and box truncations. This generalises a classic 1977 result of Fiedler, Jurkat and K\"orner  for one-variable theta sums and, in the multi-variable case, improves previous estimates obtained by Cosentino and Flaminio in 2015. Key steps in our approach are the automorphic representation of theta functions and their growth in the cusps of the underlying homogeneous space.   
\end{abstract}

\tableofcontents

\newpage

\section{Introduction}
\label{sec:intro}

Consider the exponential sum
\begin{equation}
  \label{eq:SMXxydef0}
  \theta_f(M, X, \bm{x}, \bm{y}) = \sum_{\bm{m} \in \mathbb{Z}^n} f(M^{-1} (\bm{m}+\bm{x})) \; \e\left( \tfrac{1}{2}\, \bm{m} X \transpose{\!\bm{m}} + \bm{m} \transpose{\!\bm{y}}\right) ,
\end{equation}
where $f:\mathbb{R}^n\to \mathbb{R}$ is a rapidly decaying cut-off function, $M\in\mathbb{R}_{>0}$, $X$ a real symmetric $n\times n$  matrix, and $\bm{x}, \bm{y} \in \mathbb{R}^n$ (represented as row vectors). We also use the shorthand $\e(z)=\e^{2\pi\i z}$. 

We refer to $\theta_f$ as a {\em theta sum}. If, for example $f(\bm{x})=\exp(-\pi \bm{x} P \transpose{\!\bm{x}} )$ for some positive definite matrix $P$, we obtain the classical Siegel theta series
\begin{equation}
  \label{eq:SMXxydefSIEGEL}
  \theta_f(M, X, \bm{0}, \bm{y}) = \sum_{\bm{m} \in \mathbb{Z}^n}\e\left( \tfrac{1}{2}\, \bm{m} Z \transpose{\!\bm{m}} + \bm{m} \transpose{\!\bm{y}}\right) ,
\end{equation}
with $Z= X+\i Y$ and $Y= M^{-2} P$.
If, on the other hand, $f=\chi_\mathcal{B}$ is the characteristic function of a bounded set $\mathcal{B}\subset \mathbb{R}^n$ we have the finite sum
\begin{equation}
  \label{eq:SMXxydef}
  \theta_f(M, X, \bm{x}, \bm{y}) = \sum_{\bm{m} \in \mathbb{Z}^n \cap (M\mathcal{B}-\bm{x})} \e\left( \tfrac{1}{2}\, \bm{m} X \transpose{\!\bm{m}} + \bm{m} \transpose{\!\bm{y}}\right) .
\end{equation}
In this case we will also use the notation $\theta_f=\theta_{\mathcal{B}}$.  

The following theorem, which is our first main result, gives an upper bound on the values of theta sums in the limit of large $M$, when the truncation function is in the class of complex-valued Schwartz functions $\mathcal{S}(\mathbb{R}^n)$ . 

\begin{theorem}
  \label{theorem:thetasumboundssmooth}
  Fix $f\in \mathcal{S}(\mathbb{R}^n)$ and
  let $\psi : [0,\infty ) \to [1, \infty)$ be an increasing function such that series
  \begin{equation}
    \label{eq:psiconvergence}
    \sum_{k \geq 0} \psi(k)^{-(2n+2)} 
  \end{equation}
  converges.
  Then there exists a subset $\mathcal{X}(\psi) \subset \mathbb{R}^{n\times n}_{\mathrm{sym}}$ of full Lebesgue measure such that
  for $M\geq 1$, $X \in \mathcal{X}(\psi)$, $\bm{x},\bm{y}\in \mathbb{R}^n$ we have
   \begin{equation}
    \label{eq:thetasumboundsmooth1}
    \theta_f(M, X, \bm{x}, \bm{y})  = O_{f,X}\big(M^{\frac{n}{2}} \psi(\log M) \big)  
  \end{equation}
  The implied constants in \eqref{eq:thetasumboundsmooth1} are independent of $M$, $\bm{x}$ and $\bm{y}$.
\end{theorem}

This theorem follows from a geometric representation of $\theta_f$ as an automorphic function and an application of a dynamical Borel-Cantelli lemma for flows on homogeneous spaces, theorem 1.7 in \cite{KleinbockMargulis1999}. The special case of theorem \ref{theorem:thetasumboundssmooth} for general smooth theta sums in one variable was considered in \cite{Marklof2007}. 

Upper bounds for smooth multi-variable theta sums, with an additional linear average in $X$, have played an important role in understanding the value distribution of quadratic forms, see for example the work of G\"otze \cite{Gotze2004}, Buterus, G\"otze, Hille and Margulis \cite{Buterus2019} and the first named author \cite{Marklof2002,Marklof2003}.

The second main result of this paper deals with the subtler case when $f$ is the characteristic function of a rectangular box $\mathcal{B}$. In this setting Cosentino and Flaminio \cite{ConsentinoFlaminio2015} established the bound 
\begin{equation}\label{CF}
\theta_\mathcal{B}(M, X, \bm{0}, \bm{y})= O_{X,\epsilon}\big(M^{\frac{n}{2}} (\log M)^{n+\frac{1}{2n+2}+\epsilon}\big)
\end{equation}
for the unit cube $\mathcal{B}=[0,1]^n$, any $\epsilon>0$ and almost every $X$. The following theorem improves on this by a factor of $(\log M)^n$ and produces a uniform bound for rectangular boxes of the form $\mathcal{B}=(0,b_1)\times\cdots\times(0,b_n)$, with $b_i\in\mathbb{R}_{>0}$ ranging over compacta.

\begin{theorem}
  \label{theorem:thetasumbounds}
  Fix a compact subset $\mathcal{K}\subset\mathbb{R}_{>0}^n$, and choose
  $\psi$ as in theorem \ref{theorem:thetasumboundssmooth}.
  Then there exists a subset $\mathcal{X}(\psi) \subset \mathbb{R}^{n\times n}_{\mathrm{sym}}$ of full Lebesgue measure such that
  \begin{equation}
    \label{eq:thetasumbound}
    \theta_\mathcal{B}(M, X, \bm{x}, \bm{y})  = O_X \big( M^{\frac{n}{2}} \psi(\log M) \big)
  \end{equation}
  for all $M\geq 1$, $\bm{b}=(b_1,\ldots,b_n)\in\mathcal{K}$, $X \in \mathcal{X}(\psi)$, $\bm{x},\bm{y}\in \mathbb{R}^n$. 
   The implied constants are independent of $M$, $\bm{b}$, $\bm{x}$ and $\bm{y}$.
\end{theorem}

To compare this with the bound obtained in \cite{ConsentinoFlaminio2015}, note that $\psi(t) = t^{\frac{1}{2n + 2} + \epsilon}$ satisfies (\ref{eq:psiconvergence}) and thus resulting bound (\ref{eq:thetasumbound}) indeed improves \eqref{CF} by a factor of $( \log M)^n$.
The paper \cite{ConsentinoFlaminio2015} also established the stronger bound 
\begin{equation}
\theta_\mathcal{B}(M, X, \bm{x}, \bm{y})  = O_X \big( M^{\frac{n}{2}}\big)
\end{equation}
for ``bounded-type'' $X$ that are badly approximable by rationals (these form a set of measure zero), and weaker bounds for $X$ that satisfy more relaxed Diophantine conditions. These same bounds can also be obtained from our techniques, but with no further improvements.

In the case $n = 1$ our estimate \eqref{eq:thetasumbound} matches the optimal results found by Fiedler, Jurkat and K\"orner \cite{FiedlerJurkatKorner1977}. For $n>1$, obtaining the lower bounds in these papers (which follow from the harder part of the Borel-Cantelli lemma) is more subtle, and we hope to develop an approach to this elsewhere.

The bounds for the theta sum in \eqref{theorem:thetasumbounds} and \eqref{eq:thetasumboundsmooth1} are uniform in the shift $\bm{x}$ and the linear phase $\bm{y}$.
In forthcoming work \cite{MarklofWelsh2021b} we will consider improved bounds valid for almost all $( X, \bm{x}, \bm{y}) \subset \mathbb{R}^{n\times n}_{\mathrm{sym}} \times \mathbb{R}^n \times \mathbb{R}^n$, generalising the results for $n = 1$ found in \cite{FedotovKlopp2012}.

This paper is organised as follows. We begin in section \ref{sec:groups} by recalling some basic facts about the Heisenberg group and symplectic group $\mathrm{Sp}(n, \mathbb{R})$ as well as their semi-direct product, the Jacobi group, including the Iwasawa decomposition, Haar measure, and parabolic subgroups.
We then review the Schr\"odinger and Segal-Shale-Weil representations of the Heisenberg and symplectic group, respectively.
Following the method of \cite{LionVergne1980}, these representations are used to define theta functions in section \ref{sec:theta}.

The theta functions satisfy an automorphy condition on a certain, morally-speaking discrete subgroup of the Jacobi group.
This subgroup is discussed in section \ref{sec:Gamma}.
Its projection to the symplectic group is just the integral symplectic group $\mathrm{Sp}(n, \mathbb{Z})$.
The bulk of section \ref{sec:Gamma} concerns a fundamental domain (a slight modification of Siegel's classic fundamental domain \cite{Siegel1943} based on the work of \cite{Grenier1988}) and its properties.

In section \ref{sec:theta} we define the theta functions and state their automorphy properties before analysing their asymptotic behaviour.
While for the proof of theorems \ref{theorem:thetasumboundssmooth}   and \ref{theorem:thetasumbounds}  we only need the upper bound contained in corollary \ref{corollary:upperbound}, the full asymptotics contained in theorem \ref{theorem:cuspasymptotics} may be of independent interest.
The proof of theorem \ref{theorem:cuspasymptotics} combines the properties of the fundamental domain constructed in section \ref{sec:Gamma} and basic estimates for sums over integers together with the Langlands decompositions of the maximal parabolic subgroups of the symplectic group.

We prove  theorems  \ref{theorem:thetasumboundssmooth} and \ref{theorem:thetasumbounds} in section \ref{sec:mainproof}.
Apart from the upper bound in corollary \ref{corollary:upperbound}, our method relies on upper bounds for the measure of rapidly diverging orbits under a particular one-parameter diagonal action in the symplectic group as well as (for theorem \ref{theorem:thetasumbounds}) a resolution of the singular cutoff function in (\ref{eq:SMXxydef}) using an $n$-parameter diagonal action.
The estimates for the first part are largely based on the easy part of the proof of theorem 1.7 in \cite{KleinbockMargulis1999}, which is also a main input into the method in \cite{ConsentinoFlaminio2015}.
The complications arising from the $n$-parameter flow however prevent a straightforward application of this theorem, so we instead proceed more directly with a self-contained proof.

\section{Heisenberg, symplectic, and Jacobi groups}
\label{sec:groups}

We define the  ($2n+1$)-dimensional Heisenberg group $H$ to be the set $\mathbb{R}^n \times \mathbb{R}^n \times \mathbb{R}$ with multiplication given by
\begin{equation}
  \label{eq:multiplication}
  (\bm{x}_1, \bm{y}_1, t_1)(\bm{x}_2, \bm{y}_2, t_2) = \big(\bm{x}_1 + \bm{x}_2, \bm{y}_1 + \bm{y}_2, t_1 + t_2 + \tfrac{1}{2}( \bm{y}_1 \transpose{\!\bm{x}_2} - \bm{x}_1\transpose{\!\bm{y}_2})\big). 
\end{equation}
The rank $n$ symplectic group $G = \mathrm{Sp}(n, \mathbb{R})$ is defined by
\begin{equation}
  \label{eq:Gdef}
  G = \{ g \in \mathrm{GL}(2n, \mathbb{R}) : g J_0 \transpose{\!g} = J_0 \}
\end{equation}
where
\begin{equation}
  \label{eq:J0def}
  J_0 =
  \begin{pmatrix}
    0 & -I \\
    I & 0
  \end{pmatrix}
\end{equation}
with $I$ the $n\times n$ identity matrix.
We have the alternative characterization
\begin{equation}
  \label{eq:Gdef2}
  G = \left\{
    \begin{pmatrix}
      A & B \\
      C & D
    \end{pmatrix}
    : A\transpose{\!B} = B\transpose{\!A},\ C\transpose{\!D} = D\transpose{\!C},\ A\transpose{\!D} - B \transpose{\!C} = I \right\}. 
\end{equation}

The group $G$ acts by on $H$ via
\begin{equation}
  \label{eq:GHaction}
  ( \bm{x}, \bm{y}, t)^g = ( \bm{x} A + \bm{y}C, \bm{x} B + \bm{y}D, t)
\end{equation}
where
\begin{equation}
  \label{eq:ABCD}
  g =
  \begin{pmatrix}
    A & B \\
    C & D
  \end{pmatrix}
  .
\end{equation}
Since $g$ preserves the symplectic form $J_0$ used to define the multiplication (\ref{eq:multiplication}), this action is by automorphisms, i.e. $(h_1h_2)^g = h_1^g h_2^g$.
We define the semi-direct product group $H \rtimes G$, called the Jacobi group, to be the set of all $(h, g)$, $h \in G$ and $g \in G$, with multiplication given by
\begin{equation}
  \label{eq:HGmultiplication}
  (h_1, g_1)(h_2, g_2) = (h_1 h_2^{g_1^{-1}}, g_1 g_2). 
\end{equation}

\subsection{Iwasawa decomposition and Haar measure}
\label{sec:iwasawa}

The intersection $K = G \cap \mathrm{O}(2n)$ is a maximal compact subgroup of $G$ and
\begin{equation}
  \label{eq:Kparametrization}
  Q \mapsto k(Q) = 
  \begin{pmatrix}
    \mathrm{Re}(Q) & - \mathrm{Im}(Q) \\
    \mathrm{Im}(Q) & \mathrm{Re}(Q)
  \end{pmatrix}
\end{equation}
defines an isomorphism from the unitary group $\mathrm{U}(n)$ to $K$. 
The Iwasawa decomposition of $G$ with respect to $K$ implies that any $g \in G$ can be written uniquely as
\begin{equation}
  \label{eq:Iwasawadecomp}
  g =
  \begin{pmatrix}
    A & B \\
    C & D
  \end{pmatrix}
  =
  \begin{pmatrix}
    I & X \\
    0 & I
  \end{pmatrix}
  \begin{pmatrix}
    Y^{\frac{1}{2}} & 0 \\
    0 & \transpose{Y}^{-\frac{1}{2}}
  \end{pmatrix}
  k(Q)
\end{equation}
where $X$ and $Y$ are symmetric, $Y$ is positive definite, and $Q \in \mathrm{U}(n)$.
Here we have chosen $Y^{\frac{1}{2}}$ to by upper-triangular with positive diagonal entries, and we often further decompose $Y = U V \transpose{U}$ with $U$ upper-triangular unipotent and $V$ positive diagonal.
We also note that $Y^{-\frac{1}{2}}$ is always interpreted as $(Y^{\frac{1}{2}})^{-1}$, not $ (Y^{-1})^{\frac{1}{2}}$. 
We make frequent use of the following expressions for the $X$, $Y$, and $Q$ coordinates,
\begin{align}
  \label{eq:XYQ}
  Y & = ( C\transpose{\!C} + D \transpose{\!D})^{-1} \nonumber \\
  X & = (A \transpose{\!C} + B \transpose{\!D})( C\transpose{\!C} + D \transpose{\!D})^{-1} \nonumber \\
  Q & = ( C\transpose{\!C} + D \transpose{\!D})^{-\frac{1}{2}} ( D + i C),
\end{align}
where as before $( C\transpose{\!C} + D \transpose{\!D})^{\frac{1}{2}}$ is chosen to be upper-triangular with positive diagonal entries.

The Haar measure on $G$ can be easily expressed in terms of the Iwasawa decomposition.
For
\begin{equation}
  \label{eq:gXUVQ}
  g =
  \begin{pmatrix}
    I & X \\
    0 & I
  \end{pmatrix}
  \begin{pmatrix}
    U V^{\frac{1}{2}} & 0 \\
    0 & \transpose{U}^{-1} V^{-\frac{1}{2}}
  \end{pmatrix}
  k(Q),
\end{equation}
the Haar measure $\mu$ on $G$ is given by
\begin{equation}
  \label{eq:Haarmeasure}
  \dd \mu(g) = \left( \prod_{1\leq i \leq j\leq n} \dd x_{ij} \right) \left( \prod_{1 \leq i < j \leq n} \dd u_{ij} \right) \left( \prod_{1 \leq j \leq n} v_{jj}^{-n + j -2} \dd v_{jj} \right) \dd Q.
\end{equation}
Here $\dd Q$ denotes the Haar measure on $\mathrm{U}(n)$ and $\dd x_{ij}$, $\dd u_{ij}$, $\dd v_{jj}$ are respectively the Lebesgue measures on the entries of $X$, $U$, $V$.

We note that if $g =
\begin{pmatrix}
  A & B \\
  C & D
\end{pmatrix}
$ with $D$ invertible, then we can write
\begin{equation}
  \label{eq:LUdecomp}
  g =
  \begin{pmatrix}
    I & BD^{-1} \\
    0 & I 
  \end{pmatrix}
  \begin{pmatrix}
    \transpose{\!D}^{-1} & 0 \\
    0 & D
  \end{pmatrix}
  \begin{pmatrix}
    I & 0 \\
    D^{-1} C & I
  \end{pmatrix}
  .
\end{equation}
Therefore the set of $g \in G$ having the form
\begin{equation}
  \label{eq:XATvariables}
  g =
  \begin{pmatrix}
    I & X \\
    0 & I 
  \end{pmatrix}
  \begin{pmatrix}
    A & 0 \\
    0 & \transpose{\!A}^{-1} 
  \end{pmatrix}
  \begin{pmatrix}
    I & 0 \\
    T & I
  \end{pmatrix}
\end{equation}
for $X$, $T$ symmetric and $A \in \mathrm{GL}(n, \mathbb{R})$ is open and dense in $G$.
We claim that in these coordinates we have, up to multiplication by a positive constant,
\begin{equation}
  \label{eq:XATHaar}
  \dd \mu(g) = ( \det A )^{-2n -1} \left( \prod_{i \leq j} \dd x_{ij} \right) \left( \prod_{i, j} \dd a_{ij} \right) \left( \prod_{i\leq j} \dd t_{ij} \right).
\end{equation}
where $\dd x_{ij}$, $\dd a_{ij}$, $\dd t_{ij}$ are the Lebesgue measure on the entries of $X$, $A$, $T$.

To verify (\ref{eq:XATHaar}) up to a positive constant it suffices to check that the right side is invariant under left multiplication by generators of $G$.
The invariance under matrices $
\begin{pmatrix}
  I & X_1 \\
  0 & I
\end{pmatrix}
$ with $X_1$ symmetric is obvious, and the invariance under matrices $
\begin{pmatrix}
  A_1 & 0 \\
  0 & \transpose{\!A}_1^{-1}
\end{pmatrix}
$ follows from
\begin{equation}
  \label{eq:A1invariance}
  \begin{pmatrix}
    A_1 & 0 \\
    0 & \transpose{\!A}_1^{-1}
  \end{pmatrix}
  \begin{pmatrix}
    I & X \\
    0 & I
  \end{pmatrix}
  \begin{pmatrix}
    A & 0 \\
    0 & \transpose{\!A}^{-1}
  \end{pmatrix}
  \\ =
  \begin{pmatrix}
    I & A_1 X \transpose{\!A}_1 \\
    0 & I
  \end{pmatrix}
  \begin{pmatrix}
    A_1 A & 0 \\
    0 & \transpose{\!A}_1^{-1} \transpose{\!A}^{-1}
  \end{pmatrix}
\end{equation}
and that the replacements $X \gets A_1^{-1} X \transpose{\!A}_1^{-1}$, $A \gets A_1^{-1}A$ change
\begin{equation}
  \label{eq:dXdAchange}
  \prod_{i\leq j} \dd x_{ij} \gets (\det A_1)^{-n -1} \prod_{i\leq j} \dd x_{ij}, \quad \prod_{i, j} \dd a_{ij} \gets ( \det A_1)^{-n} \prod_{i, j} \dd a_{ij}. 
\end{equation}
To verify the invariance under $
\begin{pmatrix}
  0 & -I \\
  I & 0
\end{pmatrix}
$ we may restrict further to the set of $g$ of the form (\ref{eq:XATvariables}) with $X$ invertible, as this is still an open, dense set.
We then have
\begin{multline}
  \label{eq:Jinvariance}
  \begin{pmatrix}
    0 & -I \\
    I & 0
  \end{pmatrix}
  \begin{pmatrix}
    I & X \\
    0 & I
  \end{pmatrix}
  \begin{pmatrix}
    A & 0 \\
    0 & \transpose{\!A}^{-1}
  \end{pmatrix}
  \begin{pmatrix}
    I & 0 \\
    T & I
  \end{pmatrix}
  \\ =
  \begin{pmatrix}
    I & -X^{-1} \\
    0 & I
  \end{pmatrix}
  \begin{pmatrix}
    X^{-1} A & 0 \\
    0 & \transpose{\!X} \transpose{\!A}^{-1} 
  \end{pmatrix}
  \begin{pmatrix}
    I & 0 \\
    T + \transpose{\!A} X^{-1} A & I
  \end{pmatrix}
  .
\end{multline}
The invariance then follows from the fact that the replacement $X \gets X^{-1}$ changes
\begin{equation}
  \label{eq:Xinversechange}
  \prod_{i \leq j} \dd x_{ij} \gets ( \det X)^{-n -1} \prod_{i\leq j} \dd x_{ij}. 
\end{equation}

\subsection{Parabolic subgroups}
\label{sec:parabolic}

We recall that conjugacy classes of parabolic subgroups of $G$ are in bijection with subsets of the $n$ positive simple roots, see for example section 4.5.3 of \cite{Terras1988}. 
Here we make the choice of positive simple roots $\alpha_1, \dots, \alpha_n$ where, for $1 \leq l < n$,
\begin{equation}
  \label{eq:alphaj}
  \alpha_l
  \begin{pmatrix}
    A & 0 \\
    0 & A^{-1}
  \end{pmatrix}
  = a_l a_{l+1}^{-1}
\end{equation}
and
\begin{equation}
  \label{eq:alphan}
  \alpha_n
  \begin{pmatrix}
    A & 0 \\
    0 & A^{-1}
  \end{pmatrix}
  = a_n^2.
\end{equation}
Here
\begin{equation}
  \label{eq:Vvj}
  A =
  \begin{pmatrix}
    a_1 & \cdots & 0 \\
    \vdots & \ddots & \vdots \\
    0 & \cdots & a_n
  \end{pmatrix}
\end{equation}
is positive diagonal.
See for example section 5.1 of \cite{Terras1988}.

The parabolic corresponding to a subset $L \subset \{ \alpha_1, \dots, \alpha_n\}$ is given by
\begin{equation}
  \label{eq:PJdef}
  P_L = N \bigcap_{\alpha \in L}  Z(\mathrm{ker}( \alpha ))
\end{equation}
where $Z(\mathrm{ker}(\alpha))$ is the centraliser in $G$ of the kernel of the root $\alpha$ and
\begin{equation}
  \label{eq:Ndef}
  N = \bigg\{
  \begin{pmatrix}
    U & X \transpose{U}^{-1} \\
    0 & \transpose{U}^{-1}
  \end{pmatrix}
  : \\
  U\ \mathrm{ upper\ triangular\ unipotent},\ X\ \mathrm{symmetric} \bigg\}.
\end{equation}
The maximal parabolic subgroups correspond to subsets $L$ of size $n-1$ and we denote them by $P_l$, $1\leq l \leq n$ corresponding to root $\alpha_l$ not in the set $L$.
For $1 \leq l < n$, we write an arbitrary element of $P_l$ as
\begin{equation}
  \label{eq:Pjdecomp}
  \begin{pmatrix}
    I & R_l & T_l - S_l \transpose{\!R}_l & S_l \\
    0 & I & \transpose{\!S}_l & 0 \\
    0 & 0 & I & 0 \\
    0 & 0 & -\transpose{\!R}_l & I
  \end{pmatrix}
  \begin{pmatrix}
    a_lI & 0 & 0 & 0 \\
    0 & I & 0 & 0 \\
    0 & 0 & a_l^{-1}I & 0 \\
    0 & 0 & 0 & I
  \end{pmatrix}
  \begin{pmatrix}
    U_l & 0 & 0 & 0 \\
    0 & A_l & 0 & B_l \\
    0 & 0 & \transpose{U}_l^{-1} & 0 \\
    0 & C_l & 0 & D_l
  \end{pmatrix}
\end{equation}
where $R_l$ and $S_l$ are $l \times (n-l)$ matrices, $T_l$ is $l \times l$ symmetric, $a_l > 0$, $U_l \in \mathrm{GL}(l, \mathbb{R})$ with $\det U_l = \pm 1$, and $ g_l = 
\begin{pmatrix}
  A_l & B_l \\
  C_l & D_l
\end{pmatrix}
\in \mathrm{Sp}(n - l, \mathbb{R})$.
For $ l =n$, we write an arbitrary element of $P_n$ as
\begin{equation}
  \label{eq:Pndecomp}
  \begin{pmatrix}
    I & T_n \\
    0 & I 
  \end{pmatrix}
  \begin{pmatrix}
    a_n I & 0 \\
    0 & a_n^{-1} I 
  \end{pmatrix}
  \begin{pmatrix}
    U_n & 0 \\
    0 & \transpose{\!U_n}^{-1}
  \end{pmatrix}
\end{equation}
where $T_n$ is $n\times n$ symmetric, $a_n > 0$, and $U_n \in \mathrm{GL}(n, \mathbb{R})$ with $\det U_n = \pm 1$.
The factorizations (\ref{eq:Pjdecomp}), (\ref{eq:Pndecomp}) are in fact the Langlands decompositions of $P_l$, $P_n$, which write an arbitrary element of the parabolic subgroup as a product of elements of a nilpotent subgroup, a diagonal subgroup, and a semi-simple subgroup.
For general considersations regarding the Langlands decomposition, see section 7.7 of \cite{Knapp2002}.
The author's lecture notes \cite{MarklofWelsh2021c} contain explicit calculations for the symplectic group $G$ along these lines.

\subsection{Schr\"odinger and Segal-Shale-Weil representations}
\label{sec:representations}

The Schr\"odinger representation $W$ of $H$ acts on $L^2(\mathbb{R}^n)$ by the unitary transformations
\begin{equation}
  \label{eq:Schrodingerrep}
  W(\bm{x}, \bm{y}, t) f(\bm{x}_0) = \e( -t + \tfrac{1}{2}\bm{x}\transpose{\!\bm{y}} + \bm{x}_0 \transpose{\bm{y}}) f(\bm{x}_0 + \bm{x}). 
\end{equation}
We remark that this definition of the Schr\"odinger representation differs slightly from the conventional one; they are of course unitarily equivalent.

Given $g \in G$, we obtain another representation $W^g$ of $H$ by $W^g(h) = W(h^g)$.
By the Stone-von Neumann theorem, there exists unitary operators $R(g)$ on $L^2(\mathbb{R}^n)$ such that
\begin{equation}
  \label{eq:Rgequation}
  W^g = R(g)^{-1} W R(g).
\end{equation}
The relation (\ref{eq:Rgequation}) actually defines $R(g)$ up to a scalar multiple.
Regardless of the choice of this scalar (which we make below), we have
\begin{equation}
  \label{eq:rhodef}
  R(g_1g_2) = \rho(g_1, g_2) R(g_1)R(g_2)
\end{equation}
for a nontrivial, unitary cocycle $\rho: G \times G \to \mathbb{C}$.
Thus $R$ defines a projective representation of $G$, which is called the Segal-Shale-Weil representation.
The projective representation $R$ can be extended to a true representation of the metaplectic group -- the simply connected double cover of $G$, but we do not make use of this construction.

The following proposition gives expressions for $R(g)$ for certain $g$ and on a dense subset of $L^2(\mathbb{R}^n)$.
In particular the proposition makes precise the choice of scalar multiple in our definition of $R$.

\begin{proposition}
  \label{proposition:Rproperties}
  Let $f\in\mathcal{S}(\mathbb{R}^n)$.
  Then for
  \begin{equation}
    \label{eq:gexample1}
    g =
    \begin{pmatrix}
      A & B \\
      0 & \transpose{\!A}^{-1}
    \end{pmatrix}
    ,
  \end{equation}
  we have
  \begin{equation}
    \label{eq:Rexample1}
    R(g)f(\bm{x}) = |\det A|^{\frac{1}{2}}\e\left( \frac{1}{2} \bm{x} A \transpose{\!B} \transpose{\!\bm{x}} \right) f( \bm{x} A),
  \end{equation}
  and for
  \begin{equation}
    \label{eq:gexample2}
    g = J_l =
    \begin{pmatrix}
      I & 0 & 0 & 0 \\
      0 & 0 & 0 & -I \\
      0 & 0 & I & 0 \\
      0 & I & 0 & 0 
    \end{pmatrix}
    ,
  \end{equation}
  with square blocks of size $l$, $n-l$, $l$, and $n-l$ along the diagonal, $0\leq l \leq n$, we have
  \begin{equation}
    \label{eq:Rexample2}
    R(g)f(\bm{x}) = \int_{\mathbb{R}^{n-l}} f(\bm{x}^{(1)}, \bm{y}^{(2)}) \e( - \bm{x}^{(2)} \transpose{\!\bm{y}}^{(2)}) \dd \bm{y}^{(2)}
  \end{equation}
  with $\bm{x} =
  \begin{pmatrix}
    \bm{x}^{(1)} & \bm{x}^{(2)} 
  \end{pmatrix}
  $.
  Moreover, for $g$ of the form (\ref{eq:gexample1}) and any $g' \in G$, we have $\rho(g, g') = \rho(g', g) = 1$. 
\end{proposition}

This proposition is a summary of various calculations found in \cite{LionVergne1980}.
The forthcoming lecture notes by the authors \cite{MarklofWelsh2021c} will give self-contained proofs. 

We remark that together with the Bruhat decomposition
\begin{equation}
  \label{eq:Bruhatdecomp}
  G =  \bigcup_{0\leq l \leq n} P_n J_l P_n
\end{equation}
where $P_n$ is the maximal parabolic subgroup (\ref{eq:Pndecomp}), proposition \ref{proposition:Rproperties} allows one to compute $R(g)$ for any $g \in G$.
For example, for
\begin{equation}
  \label{eq:gexample3}
  g =
  \begin{pmatrix}
    A & B \\
    C & D
  \end{pmatrix}
  =
  \begin{pmatrix}
    I & AC^{-1} \\
    0 & I
  \end{pmatrix}
  \begin{pmatrix}
    0 & -I \\
    I & 0
  \end{pmatrix}
  \begin{pmatrix}
    C & D \\
    0 & \transpose{C}^{-1}
  \end{pmatrix}
\end{equation}
with $C$ invertible, we have
\begin{equation}
  \label{eq:Rexample3}
  R(g)f(\bm{x}) = | \det C|^{- \frac{1}{2}} \e\left( \frac{1}{2} \bm{x} A C^{-1} \transpose{\!\bm{x}}\right)  \\
  \int_{\mathbb{R}^n} f(\bm{y}) \e\left( \frac{1}{2} \bm{y} C^{-1}D \transpose{\!\bm{y}} - \bm{x} \transpose{\!C}^{-1} \transpose{\!\bm{y}}\right) \dd \bm{y}.
\end{equation}

\section{The subgroups $\Gamma$ and $\tilde{\Gamma}$}
\label{sec:Gamma}

We denote by $\Gamma$ the discrete subgroup $\Gamma = \mathrm{Sp}(n ,\mathbb{Z}) \subset G$. 
For
\begin{equation}
  \label{eq:ABCDdef2}
  \gamma = 
  \begin{pmatrix}
    A & B \\
    C & D
  \end{pmatrix}
  \in \Gamma,
\end{equation}
we set $h_\gamma = (\bm{r}, \bm{s}, 0) \in H$ where the entries or $\bm{r}$ are $0$ or $\frac{1}{2}$ depending on whether the corresponding diagonal entry of $C\transpose{D}$ is even or odd, and the entries of $\bm{s}$ are $0$ or $\frac{1}{2}$ depending on whether the corresponding diagonal entry of $A\transpose{B}$ is even or odd.
We now define the group $\tilde{\Gamma} \subset H \rtimes G$ by
\begin{equation}
  \label{eq:Gammatildedef}
  \tilde{\Gamma} = \{ ((\bm{m}, \bm{n}, t) h_\gamma ,  \gamma) \in H \rtimes G : \gamma \in \Gamma, \bm{m} \in \mathbb{Z}^n, \bm{n} \in \mathbb{Z}^n, t \in \mathbb{R}\}.
\end{equation}
We note that this is a subgroup of $G$ because, modulo left multiplication by elements $(\bm{m}, \bm{n}, t)\in H$ with $\bm{m}, \bm{n} \in \mathbb{Z}^n$ and $t \in \mathbb{R}$, we have $h_{\gamma_1 \gamma_2} = h_{\gamma_1} h_{\gamma_2}^{\gamma_1^{-1}}$ for any $\gamma_1, \gamma_2 \in \Gamma$.
Indeed, we have that $2\bm{r}$ and $2\bm{s}$, where $h_{\gamma_1 \gamma_2} = (\bm{r}, \bm{s},0)$, has the same parity as the diagonal entries of
\begin{equation}
  \label{eq:hgamma1gamma2}
  (C_1 A_2 + D_1 C_2) \transpose{(C_1 B_2 + D_1 D_2)}
\end{equation}
and
\begin{equation}
  \label{eq:hgamma1gamma2s}
  (A_1 A_2 + B_1 D_2) \transpose{(A_1 B_2 + B_1 D_2)},
\end{equation}
which, in view of $A_2 \transpose{D}_2 - B_2 \transpose{C_2} = I$, have the same parity as the diagonal entries of
\begin{equation}
  \label{eq:hgamma1gamma2r1}
  C_1 A_2 \transpose{B}_2 \transpose{C}_1 + D_1 C_2 \transpose{D}_2 \transpose{D}_1 + C_1 \transpose{D}_1
\end{equation}
and
\begin{equation}
  \label{eq:hgamma1gamma2s1}
  A_1 A_2 \transpose{B}_2 \transpose{A}_1 + B_1 C_2 \transpose{D}_2 \transpose{B}_1 + A_1 \transpose{B}_1.
\end{equation}
On the other hand, we have
\begin{equation}
  \label{eq:hgamma1hgamma2}
  h_{\gamma_1} h_{\gamma_2}^{\gamma_1^{-1}} = \left( \bm{r}_1 + \bm{r}_2 \transpose{D}_1 - \bm{s}_2 \transpose{C}_1, \bm{s}_1 - \bm{r}_2 \transpose{B}_1 + \bm{s}_2 \transpose{A}_1, * \right),
\end{equation}
where $\gamma_{\gamma_j} = ( \bm{r}_j, \bm{s}_j, 0)$.
The entries of two times the vectors on the right of (\ref{eq:hgamma1hgamma2}) have the same parity as the diagonal entries of (\ref{eq:hgamma1gamma2r1}) and (\ref{eq:hgamma1gamma2s1}), as claimed.

We say that a closed set $\mathcal{D} \subset G$ is a fundamental domain for $\Gamma \backslash G$ if
\begin{enumerate}
\item for all $g \in G$ there exists $\gamma \in \Gamma$ such that $\gamma g \in \mathcal{D}$ and 
\item if for $g \in \mathcal{D}$ there is a non-identity $\gamma \in \Gamma$ such that $\gamma g \in \mathcal{D}$, then $g$ is contained in the boundary of $\mathcal{D}$. 
\end{enumerate}
Following Siegel \cite{Siegel1943}, we define $\mathcal{D}$ to be the set of all
\begin{equation}
  \label{eq:gXYQ}
  g =
  \begin{pmatrix}
    I & X \\
    0 & I
  \end{pmatrix}
  \begin{pmatrix}
    Y^{\frac{1}{2}} & 0 \\
    0 & \transpose{Y}^{-\frac{1}{2}}
  \end{pmatrix}
  k(Q) \in G
\end{equation}
such that
\begin{enumerate}
\item $| \det ( C ( X + \mathrm{i} Y) + D) | \geq 1$ for all $
  \begin{pmatrix}
    A & B \\
    C & D
  \end{pmatrix}
  \in \Gamma $, 
\item $Y \in \mathcal{D}'$, a fundamental domain for the action of $\mathrm{GL}(n, \mathbb{Z})$ on $n\times n$ positive symmetric matrices, and 
\item $|x_{ij}| \leq \frac{1}{2}$, where $x_{ij}$ are the entries of $X$. 
\end{enumerate}
We note that since (\ref{eq:XYQ}) implies that
\begin{equation}
  \label{eq:Ygammag}
  Y(\gamma g) = \transpose{\! ( C ( X - \mathrm{i}Y) + D)}^{-1} Y ( C(X + \mathrm{i}Y) + D)^{-1},
\end{equation}
the first condition implies that for $g \in \mathcal{D}$, $| \det (Y(g))| \geq |\det (Y(\gamma g))|$ for all $\gamma \in \Gamma$. 
We also note that Siegel chooses $\mathcal{D}'$ to be the set of positive definite symmetric $Y$ such that $Y^{-1}$ is in Minkowski's classical fundamental domain.
However, here we choose $\mathcal{D}'$ to be the set of $Y$ such that $Y^{-1}$ is in Grenier's fundamental domain, see \cite{Grenier1988} and \cite{Terras1988}.

Following \cite{Grenier1988} and \cite{Terras1988}, we define $\mathcal{D}' = \mathcal{D}'_n$ recursively as follows.
We set $\mathcal{D}_1' = \{ y > 0 \}$ and 
\begin{equation}
  \label{eq:n1funddom}
  \mathcal{D}'_2 =  \bigg\{
    \begin{pmatrix}
      1 & r_1 \\
      0 & 1
    \end{pmatrix}
    \begin{pmatrix}
      v_1 & 0 \\
      0 & v_2 
    \end{pmatrix}
    \begin{pmatrix}
      1 & 0 \\
      r_1 & 1
    \end{pmatrix}
  : \\
  r_1^2 + \frac{v_1}{v_2} \geq 1,\ 0\leq u \leq \frac{1}{2},\ v_1, v_2 > 0 \bigg\},
\end{equation}
the standard fundamental domain for $\mathrm{GL}(2, \mathbb{Z})$.
For $n > 2$ we define $\mathcal{D}_n'$ to be the set of 
\begin{equation}
  \label{eq:Yuv1}
  Y =
  \begin{pmatrix}
    1 & \bm{r}_1 \\
    0 & I 
  \end{pmatrix}
  \begin{pmatrix}
    v_1 & 0 \\
    0 & Y_1
  \end{pmatrix}
  \begin{pmatrix}
    1 & 0 \\
    \transpose{\!\bm{r}_1} & I
  \end{pmatrix}
\end{equation}
such that
\begin{enumerate}
\item $v_1 = v_1(Y) \geq v_1(A Y \transpose{\!A})$ for all $A \in \mathrm{GL}(n, \mathbb{Z})$, 
\item $Y_1(Y) \in \mathcal{D}_{n-1}'$, and 
\item $|r_j| \leq \frac{1}{2}$ and $ 0 \leq r_1 \leq \frac{1}{2}$, where $r_j$ are the entries of $\bm{r}_1$. 
\end{enumerate}

This is proven to be a fundamental domain in \cite{Grenier1988} and \cite{Terras1988}.
In general, the motivation for using this fundamental domain is the box-shaped cusp, but here the primary advantage is its recursive definition, which we make frequent use of below.
We remark that one can construct a fundamental domain for $\Gamma \backslash G$ with a box-shaped cusp by maximising $v_1$ over all of $\Gamma$, not just $\mathrm{GL}(n, \mathbb{Z})$.
This approach is utilised in the second paper in this series \cite{MarklofWelsh2021b}. 
However we do not need this feature here, and in fact maximising the determinant in the fundamental domain as we have done is useful in what follows, see the proofs of lemmas \ref{lemma:heightcontinuity} and \ref{lemma:heightcontinuity}. 
  
The following proposition records some useful properties of $\mathcal{D}$.
\begin{proposition}
  \label{proposition:Dnproperties}
  Let $g \in \mathcal{D}$ and write
  \begin{equation}
    \label{eq:gVrewrite}
    g =
    \begin{pmatrix}
      I & X \\
      0 & I
    \end{pmatrix}
    \begin{pmatrix}
      Y^{\frac{1}{2}} & 0 \\
      0 & Y^{-\frac{1}{2}}
    \end{pmatrix}
    k(Q), \quad Y =
    \begin{pmatrix}
      1 & \bm{r}_1 \\
      0 & I
    \end{pmatrix}
    \begin{pmatrix}
      v_1 & 0 \\
      0 & Y_1
    \end{pmatrix}
    \begin{pmatrix}
      1 & 0 \\
      \transpose{\!\bm{r}_1} & I
    \end{pmatrix}
    ,
  \end{equation}
  and also 
  \begin{equation}
    \label{eq:Yrewrite}
    Y = U V \transpose{U}, \quad  V =
    \begin{pmatrix}
      v_1 & \cdots & 0 \\
      \vdots & \ddots & \vdots \\
      0 & \cdots & v_n
    \end{pmatrix}
    .
  \end{equation} 
  Then we have
  \begin{enumerate}
  \item $v_n \geq \frac{\sqrt{3}}{2}$ and $v_j \geq \frac{3}{4} v_{j+1}$ for $1\leq j \leq n-1$ and
  \item for all $\bm{x} =
    \begin{pmatrix}
      x^{(1)} & \bm{x}^{(2)}
    \end{pmatrix}
    \in \mathbb{R}^n$,
    \begin{equation}
      \label{eq:approxortho}
      \bm{x} Y \transpose{\!\bm{x}} \asymp_n v_1 (x^{(1)})^2 + \bm{x}^{(2)} Y_1 \transpose{\!\bm{x}}^{(2)}.
    \end{equation}
  \end{enumerate}
\end{proposition}

\begin{proof}
  We apply $|\det( C (X + \mathrm{i}Y) + D)|^2 \geq 1$ for
  \begin{equation}
    \label{eq:CDexample}
    C =
    \begin{pmatrix}
      0 & \cdots & 0 & 0 \\
      \vdots & \ddots & \vdots & \vdots \\
      0 & \cdots & 0 & 0 \\
      0 & \cdots & 0 & 1       
    \end{pmatrix}
    ,\quad D = 
    \begin{pmatrix}
      1 & \cdots & 0 & 0\\
      \vdots & \ddots & \vdots & \vdots \\
      0 & \cdots & 1 & 0 \\
      0 & \cdots & 0 & 0 
    \end{pmatrix}
    ,
  \end{equation}
  which may be completed to $
  \begin{pmatrix}
    A & B \\
    C & D
  \end{pmatrix}
  \in \Gamma$ with $A = D$ and $B = C$.
  We have
  \begin{equation}
    \label{eq:CDexampledet}
    | \det( C(X + \mathrm{i}Y) + D) |^2 = x_{nn}^2 + v_n^2,
  \end{equation}
  where $x_{nn}$ is the $(n, n)$ entry of $X$.
  Since the entries of $X$ are at most a half in absolute value, $v_n^2 \geq 1 - x_{nn}^2 \geq \frac{3}{4}$ as required.

  We have
  \begin{equation}
    \label{eq:Yinverse}
    Y =
    \begin{pmatrix}
      1 & \bm{r}_1 \\
      0 & I 
    \end{pmatrix}
    \begin{pmatrix}
      v_1 & 0 \\
      0 & Y_1
    \end{pmatrix}
    \begin{pmatrix}
      1 & 0 \\
      -\bm{r}_1 & I
    \end{pmatrix}
    \in \mathcal{D}'
  \end{equation}
  and we note that to demonstrate $v_j \geq \frac{3}{4} v_{j+1}$, it suffices to consider $j =1$ by the inductive construction of $\mathcal{D}'$.
  We apply the minimality of $v_1^{-1}$ for an element $\gamma \in \mathrm{GL}(n, \mathbb{Z})$ having first row $
  \begin{pmatrix}
    0 & 1 & 0 & \cdots & 0
  \end{pmatrix}
  $.
  We find that
  \begin{equation}
    \label{eq:v1min}
    v_1^{-1} \leq v_1^{-1}r^2+ v_2^{-1},
  \end{equation}
  where $r$ is the first entry of $\bm{r}_1$.
  Since $|r| \leq \frac{1}{2}$, it follows that $v_1 \geq \frac{3}{4} v_2$.
  
  To demonstrate the second part of the proposition, we let $\bm{y}_1, \dots, \bm{y}_n$ denote the rows of
  \begin{equation}
    \label{eq:Ysquareroot}
    Y^{\frac{1}{2}} =
    \begin{pmatrix}
      1 & \bm{r}_1 \\
      0 & I 
    \end{pmatrix}
    \begin{pmatrix}
      v_1^{\frac{1}{2}} & 0 \\
      0 & Y_1^{\frac{1}{2}}
    \end{pmatrix}
    .
  \end{equation}
  Setting $\bm{y} = x_2 \bm{y}_2 + \cdots + x_n \bm{y}_n$, where the $x_j$ are the entries of $\bm{x}$, our aim is to prove that for some constants $0 < c_1 < 1 < c_2$ depending only on $n$,
  \begin{equation}
    \label{eq:rbound}
    c_1 \left( ||\bm{y}_1||^2 x_1^2 + || \bm{y} ||^2 \right) \leq || x_1 \bm{y}_1 + \bm{y} ||^2 \leq c_2 \left( ||\bm{y}_1||^2 x_1^2 + || \bm{y} ||^2 \right) ,
  \end{equation}
  from which the lower bound in (\ref{eq:approxortho}) follows as $|| \bm{y}_1 ||^2 \geq v_1$.
  The upper bound in (\ref{eq:approxortho}) follows from $v_1 \gg || \bm{y}_1 ||^2$, which is verified below, see (\ref{eq:r1normbound}). 
  Expanding the expression in the middle of (\ref{eq:rbound}), we find that it is enough to show that
  \begin{equation}
    \label{eq:leftsideexpand}
    2 | x_1 \bm{y}_1 \transpose{\!\bm{y}} | \leq (1 - c_1) \left( || \bm{y}_1||^2  x_1^2 + || \bm{y} ||^2 \right),
  \end{equation}
  and
  \begin{equation}
    \label{eq:middleexpand2}
    2 | x_1 \bm{y}_1 \transpose{\!\bm{y}} | \leq (c_2 -1) \left( || \bm{y}_1||^2  x_1^2 + || \bm{y} ||^2 \right).  
  \end{equation}
  The upper bound (\ref{eq:middleexpand2}) is trivial if $c_2 = 2$, and the upper bound (\ref{eq:leftsideexpand}) would follow from
  \begin{equation}
    \label{eq:weakerinequality}
    | \bm{y}_1 \transpose{\!\bm{y}} | \leq (1 - c_1) || \bm{y}_1||\; || \bm{y} ||.
  \end{equation}

  We let $0 < \phi_1 < \pi$ denote the angle between $\bm{y}_1$ and $\bm{y}$ and $ 0 < \phi_2 < \frac{\pi}{2}$ denote the angle between $\bm{y}_1$ and the hyperplane $\mathrm{span} (\bm{y}_2, \dots, \bm{y}_n)$.
  We have $\phi_2 \leq \mathrm{min}(\phi_1, \pi - \phi_1)$, and so $| \cos \phi_1 | \leq | \cos \phi_2 |$.
  We bound $ \cos \phi_2$ away from $1$ by bounding $\sin \phi_2$ away from $0$.

  We have
  \begin{equation}
    \label{eq:sinphi2bound}
    |\sin \phi_2 | = \frac{|| \bm{y}_1 \wedge \cdots \wedge \bm{y}_n || }{|| \bm{y}_1 ||\; || \bm{y}_2 \wedge \cdots \wedge \bm{y}_n ||} = \frac{v_1^{\frac{1}{2}}}{|| \bm{y}_1 ||},
  \end{equation}
  so it suffices to show that $v_1^{\frac{1}{2}} \gg || \bm{y}_1 ||$.
  Here $\wedge$ denotes the usual wedge product on $\mathbb{R}^n$ and the norm on $\bigwedge^k \mathbb{R}^n$ is given by
  \begin{equation}
    \label{eq:wedgenorm}
    || \bm{a}_1 \wedge \cdots \wedge \bm{a}_k ||^2 = \det
    \begin{pmatrix}
      \bm{a}_1 \\
      \vdots \\
      \bm{a}_k
    \end{pmatrix}
    \begin{pmatrix}
      \transpose{\bm{a}}_1 & \cdots & \transpose{\bm{a}}_k
    \end{pmatrix}
    .
  \end{equation}
  Using the inductive construction of $\mathcal{D}'$ and the fact that the entries of $\bm{r}_1(Y), \bm{r}_1(Y_1), \dots$ are at most $\frac{1}{2}$ in absolute value, we observe that $U$ has entries bounded by a constant depending only on $n$.
  We find that
  \begin{equation}
    \label{eq:r1normbound}
    || \bm{y}_1 ||^2 \ll v_1 + \cdots + v_n \ll v_1
  \end{equation}
  with the implied constant depending on $n$. 
\end{proof}

\section{Theta functions and asymptotics}
\label{sec:theta}

Following \cite{LionVergne1980}, for $f\in\mathcal{S}(\mathbb{R}^n)$ we define the theta function $\Theta_f : H \rtimes G \mapsto \mathbb{C}$ by
\begin{equation}
  \label{eq:thetadef}
  \Theta_f( h, g) = \sum_{\bm{m} \in \mathbb{Z}^n} (W(h)R(g)f)( \bm{m}).
\end{equation}
Setting $h = (\bm{x}, \bm{y}, t)$,
\begin{equation}
  \label{eq:gXY}
  g =
  \begin{pmatrix}
    I & X \\
    0 & I
  \end{pmatrix}
  \begin{pmatrix}
    Y^{\frac{1}{2}} & 0 \\
    0 & \transpose{Y}^{-\frac{1}{2}}
  \end{pmatrix}
  k(Q),
\end{equation}
and $f_Q = R(k(Q))f$, we have from (\ref{eq:Schrodingerrep}) and (\ref{eq:Rexample1}) that
\begin{multline}
  \label{eq:explicttheta}
  \Theta_f(h, g) = (\det Y)^{\frac{1}{4}} \e\left(-t + \frac{1}{2} \bm{x} \transpose{\!\bm{y}}\right) \\
  \sum_{\bm{m} \in \mathbb{Z}^n} f_Q \left( (\bm{m} + \bm{x})Y^{\frac{1}{2}}\right) \e\left( \frac{1}{2} (\bm{m} + \bm{x})X \transpose{\!(\bm{m} + \bm{x})} + \bm{m} \transpose{\!\bm{y}}\right).
\end{multline}
Thus for $f(\bm{x}) = \exp(-\pi \bm{x} \transpose{\!\bm{x}})$, $Q = I$, and $h = (0,0,0)$, we recover $(\det Y)^{\frac{1}{4}}$ times the classical Siegel theta series that is holomorphic in $Z = X + iY$.\footnote{
  In fact, one can show that $f(\bm{x}) = \exp(- \pi \bm{x}\transpose{\!\bm{x}})$ is a simultaneous eigenfunction of all the $R(k(Q))$, $Q \in \mathrm{U}(n)$, see \cite{LionVergne1980}.
This together with theorem \ref{theorem:thetaautomorphy} establishes the automorphy of the holomorphic theta function.}

The following theorem establishes the automorphy of $\Theta_f$ under $\tilde{\Gamma}$, which we recall is defined at the beginning of section \ref{sec:Gamma}. 

\begin{theorem}
  \label{theorem:thetaautomorphy}
  For all $(uh_\gamma, \gamma) \in \tilde{\Gamma}$ and $(h, g) \in H \rtimes G$, there is a complex number $\varepsilon(\gamma)$ of modulus $1$ such that
  \begin{equation}
    \label{eq:thetaautomorphy}
    \Theta_f\left((uh_\gamma, \gamma)(h, g)\right) = \varepsilon(\gamma) \rho(\gamma, g) \e\left( -t + \frac{1}{2} \bm{m}\transpose{\!\bm{n}}\right) \Theta_f(h, g),
  \end{equation}
  where $u = (\bm{m}, \bm{n}, t)$.
\end{theorem}

This theorem is proved in \cite{LionVergne1980} but with $\tilde{\Gamma}$ replaced by a finite index subgroup.
The automorphy under the full group $\tilde{\Gamma}$ is proved in \cite{Mumford1983}, however only for the special function $f(\bm{x}) = \exp( -\pi \bm{x} \transpose{\!\bm{x}})$.
In \cite{LionVergne1980} it is shown that this function is an eigenfunction for all of the operators $R(k(Q))$, $Q \in \mathrm{U}(n)$.
Moreover, it can be seen from the theory built there that the automorphy for any Schwartz function follows from that for $\exp( -\pi \bm{x} \transpose{\!\bm{x}})$.
A self-contained proof along the lines of \cite{LionVergne1980} is presented in forthcoming notes by the authors \cite{MarklofWelsh2021c}. 
We also remark that $\varepsilon(\gamma)$ can be expressed as a kind of Gauss sum as shown in \cite{LionVergne1980} and the author's notes, but we do not make use of this here. 

We recall that for $Q \in \mathrm{U}(n)$ and $f \in L^2(\mathbb{R}^n)$, we let $f_Q = R(k(Q)) f$.
The following lemma states that if $f$ is a Schwartz function, then the $f_Q$ are ``uniformly Schwartz.''

\begin{lemma}
  \label{lemma:uniformSchwartz}
  Let $f \in\mathcal{S}(\mathbb{R}^n)$.
  Then for all $A > 0$ and multi-indices $\alpha \geq 0$, there exist constants $c_f(\alpha, A)$ such that for all $Q \in \mathrm{U}(n)$,
  \begin{equation}
    \label{eq:cfalphadef}
    \left | \left( \frac{\partial}{\partial \bm{x}}\right)^\alpha f_Q(\bm{x}) \right| \leq c_f(\alpha, A)(1 + ||\bm{x}||)^{-A}. 
  \end{equation}
\end{lemma}

\begin{proof}
  Since $f$ is Schwartz, so are the Fourier transforms of $f$ with respect to any subset of the variables.
  For a subset $S \subset \{1, \dots, n\}$, multi-index $\alpha \geq 0$, and $A > 0$, we let $c_f^S(\alpha, A)$ be constants such that
  \begin{equation}
    \label{eq:cfSdef}
    \left | \left( \frac{\partial}{\partial \bm{x}}\right)^\alpha f^S(\bm{x}) \right| \leq c_f^S(\alpha, A)(1 + ||\bm{x}||)^{-A}
  \end{equation}
  where $f^S$ is the Fourier transform of $f$ in the variables having indices in $S$.

  We now consider $f_Q$ for $Q \in \mathrm{U}(n)$ diagonal with the first $n-l$ entries $1$ and the last $l$ entries $\e^{\mathrm{i} \phi_j}$ with $0 < \phi_j < \pi$.
  We let $S \subset \{1, \dots, n\}$ be the set of indices $j$, $n-l < j \leq n$, such that $\phi_j \in \left(0, \frac{\pi}{4}\right) \cup \left( \frac{3\pi}{4}, \pi \right)$ and we write $Q = Q'Q_S$ where $Q_S$ is diagonal with $(j,j)$ entry $\mathrm{i}$ if $j \in S$ and $1$ if $j \not\in S$.
  We have
  \begin{equation}
    \label{eq:fQQprime}
    f_Q = \rho(k(Q'), k(Q_S)) R(k(Q')) R(k(Q_S)) f,
  \end{equation}
  and we recall that $|\rho(k(Q'), k(Q_S))| = 1$.
  
  We write
  \begin{equation}
    \label{eq:RkUform}
    k(Q') =
    \begin{pmatrix}
      I & 0 & 0 & 0 \\
      0 & D & 0 & -C \\
      0 & 0 & I & 0 \\
      0 & C & 0 & D
    \end{pmatrix}
    ,
  \end{equation}
  with $C$, $D$ diagonal, the entries of $C$ being $\cos \phi_j$ or $\sin \phi_j$ depending on whether $j \in S$ or not, the entries of $D$ being $- \sin \phi_j$ or $\cos \phi_j$ depending on whether $j \in S$ or not.
  We note that the entries of $C$ are at least $\frac{1}{\sqrt{2}}$. 
    
  Writing
  \begin{equation}
    \label{eq:QBruhat}
    k(Q') =
    \begin{pmatrix}
      I & 0 & 0 & 0 \\
      0 & I & 0 & C^{-1}D \\
      0 & 0 & I & 0 \\
      0 & 0 & 0 & I
    \end{pmatrix}
    \begin{pmatrix}
      I & 0 & 0 & 0 \\
      0 & 0 & 0 & -I \\
      0 & 0 & I & 0 \\
      0 & I & 0 & 0 
    \end{pmatrix}
    \begin{pmatrix}
      I & 0 & 0 & 0 \\
      0 & C & 0 & D \\
      0 & 0 & I & 0 \\
      0 & 0 & 0 & C^{-1}
    \end{pmatrix}
    ,
  \end{equation}
  using proposition \ref{proposition:Rproperties}, and noting that $R(k(Q_S)) = f^S$, we compute
  \begin{multline}
    f_Q( \bm{x}^{(1)}, \bm{x}^{(2)}) = \rho(k(Q'), k(Q_S)) |\det C |^{-\frac{1}{2}} \int_{\mathbb{R}^l} f^S(\bm{x}^{(1)}, \bm{y}^{(2)}) \\
    \e \left( \frac{1}{2} \bm{x}^{(2)} C^{-1}D \transpose{\!\bm{x}}^{(2)} - \bm{x}^{(2)} C^{-1} \transpose{\!\bm{y}}^{(2)} + \frac{1}{2} \bm{y}^{(2)} C^{-1}D \transpose{\!\bm{y}}^{(2)} \right) \dd \bm{y}^{(2)}. 
  \end{multline}
  Now as the entries of $C$ are between $\frac{1}{\sqrt{2}}$ and $1$, and the entries of $D$ are at most $\frac{1}{\sqrt{2}}$ in absolute value, integration by parts and (\ref{eq:cfSdef}) shows that
  \begin{equation}
    \label{eq:Schwartznorm}
    \left| \left( \frac{\partial}{\partial \bm{x}} \right)^\alpha f_Q(\bm{x}) \right| \ll (1 + || \bm{x}||)^{-A},
  \end{equation}
  with implied constant depending on $f$, $S$, $\alpha$, and $A$.

  We observe that for real orthogonal $Q_1$, $R(k(Q_1))f(\bm{x}) = f(\bm{x}Q_1)$, so (\ref{eq:Schwartznorm}) implies
  \begin{equation}
    \label{eq:Schwartznorm1}
    \left| \left( \frac{\partial}{\partial \bm{x}} \right)^\alpha f_{Q_1 Q Q_2}\bm{x}) \right| \ll (1 + || \bm{x}||)^{-A},
  \end{equation}
  for any orthogonal $Q$.
  It now suffices to show that any unitary matrix $Q_0$ can be written as $Q_1 Q Q_2$ with $Q$ having the special form above and $Q_1$, $Q_2$ real orthogonal.
  Writing $Q_0 \transpose{\!Q}_0 = X + i Y$ with $X$, $Y$ real and symmetric, we note that since $Q_0\transpose{\!Q}_0$ is unitary, $X^2 + Y^2 + i(XY - YX)$ is the identity.
  It follows that $X$, $Y$ commute, and thus can be simultaneously diagonalized by an orthogonal matrix $Q_1$.
  We have $Q_0 \transpose{\!Q}_0 = Q_1 Q^2 \transpose{\!Q}_1$ with $Q$ diagonal, and so $Q_2 = Q^{-1} Q_1^{-1} Q_0$ is orthogonal.
  Finally, we may permute the diagonal entries of $Q$ and change their signs so that the special form above holds. 
\end{proof}

We now turn to analysing the behaviour of the theta function $\Theta_f$, $f$ a Schwartz function, in the cusp of $\tilde{\Gamma} \backslash H \rtimes G$.
We repeatedly use the easy bounds recorded in the following lemma.

\begin{lemma}
  \label{lemma:integersums}
  For real numbers $A > \frac{1}{2}$, $|x|\leq \frac{1}{2}$ and $v, y>0$, we have
  \begin{equation}
    \label{eq:integersum1}
    \sum_{\substack{ m \in \mathbb{Z} \\ m \neq 0}} (v(m + x)^2 + y)^{-A} \ll_A v^{-\frac{1}{2}}(v + y)^{-A + \frac{1}{2}}.
  \end{equation}
  and, if in addition $v \leq a y$ with $a > 0$,
  \begin{equation}
    \label{eq:integersum2}
    \sum_{m \in \mathbb{Z}} (v(m + x)^2 + y)^{-A} \ll_{a,A} v^{-\frac{1}{2}} y^{-A}
  \end{equation}
  for $A > \frac{1}{2}$. 
\end{lemma}

\begin{proof}
  We have
  \begin{equation}
    \label{eq:integersum1bound}
    \sum_{\substack{ m \in \mathbb{Z} \\ m \neq 0}} (v(m + x)^2 + y)^{-A} \leq \sum_{0 < |m| \leq \sqrt{\frac{y}{v}}} y^{-A} + \sum_{|m| > \sqrt{\frac{y}{v}}} v^{-A} |m + x|^{-2A}.
  \end{equation}
  The first sum here is $0$ if $y < v$, otherwise it is at most $v^{-\frac{1}{2}} y^{-A + \frac{1}{2}}$.
  The second sum is at most
  \begin{equation}
    \label{eq:integersum1bound2}
    2^{2A} v^{-A} \sum_{|m| > \sqrt{\frac{y}{v}}} |m|^{-2A},
  \end{equation}
  which is $\ll_A v^{-A}$ if $y < v$ and $\ll_A v^{-\frac{1}{2}} y^{-A + \frac{1}{2}}$ otherwise.
  The estimate (\ref{eq:integersum1}) now follows as $v^{-A}$, respectively $v^{-\frac{1}{2}} y^{-A + \frac{1}{2}}$, is $\ll v^{-\frac{1}{2}}( v + y)^{-A + \frac{1}{2}}$ if $y < v$, respectively $y\geq v$.

  Turning to (\ref{eq:integersum2}), we have
  \begin{equation}
    \label{eq:integersum2bound}
    \sum_{m \in \mathbb{Z}} ( v(m + x)^2 + y)^{-A} \leq \sum_{|m| \leq \sqrt{\frac{ay}{v}}} y^{-A} + \sum_{|m| >  \sqrt{\frac{ay}{v}}}v^{-A} |m + x|^{-2A}.
  \end{equation}
  The first sum here is $\ll_a v^{-\frac{1}{2}} y^{-A + \frac{1}{2}}$, while the second sum is at most
  \begin{equation}
    \label{eq:integersum2bound2}
    2^{2A} v^{-A}  \sum_{|m| > \sqrt{\frac{ay}{v}}} |m|^{-2A} \ll_{a,A} v^{-\frac{1}{2}} y^{-A + \frac{1}{2}},
  \end{equation}
  so (\ref{eq:integersum2}) follows immediately.
\end{proof}

The following theorem, while a little complicated, gives an asymptotic formula for $\Theta_f(h, g)$ as $g \to \infty$ inside the fundamental domain $\mathcal{D}$.
We describe the relevant neighbourhoods of $\infty$ using the Langlands decomposition (\ref{eq:Pjdecomp}) of the parabolic subgroups $P_l$ with $1\leq l < n$, see (\ref{eq:Langlands}). 
The semi-simple part of the Langlands decomposition of this parabolic is a copy of $\mathrm{Sp}(n - l, \mathbb{R})$, and our asymptotic formula for $\Theta_f$ has a theta function associated to $\mathrm{Sp}(n -l, \mathbb{R})$ for a main term, see (\ref{eq:cuspasymptotics}).

\begin{theorem}
  \label{theorem:cuspasymptotics}
  Let $f\in \mathcal{S}(\mathbb{R}^n)$, $g \in \mathcal{D}$, and $h = (\bm{x}, \bm{y}, t) \in H$ with the entries of $\bm{x}$ and $\bm{y}$ all at most $\frac{1}{2}$ in absolute value.
  For $1\leq l \leq n$ we write
  \begin{multline}
    \label{eq:Langlands}
    g = \begin{pmatrix}
      I & 0 & T_l & S_l \\
      0 & I & \transpose{\!S}_l & 0 \\
      0 & 0 & I & 0 \\
      0 & 0 & 0 & I
    \end{pmatrix}
    \begin{pmatrix}
      U_l & R_l & 0 & 0 \\
      0 & I & 0 & 0 \\
      0 & 0 & \transpose{\!U_l}^{-1} & 0 \\
      0 & 0 & - \transpose{\!R_l}\transpose{\!U_l}^{-1} & I
    \end{pmatrix}
    \\
    \begin{pmatrix}
      I & 0 & 0 & 0 \\
      0 & I & 0 & X_l \\
      0 & 0 & I & 0 \\
      0 & 0 & 0 & I
    \end{pmatrix}
    \begin{pmatrix}
      V_l^{\frac{1}{2}} & 0 & 0 & 0 \\
      0 & Y_l^{\frac{1}{2}} & 0 & 0 \\
      0 & 0 & V_l^{\frac{1}{2}} & 0 \\
      0 & 0 & 0 & \transpose{\!Y_l}^{-\frac{1}{2}}
    \end{pmatrix}
    k(Q)
  \end{multline}
  where $R_l$, $S_l$ are $l \times (n-l)$ matrices, $T_l$ is $l\times l$ symmetric, $U_l$ is $l\times l$ upper-triangular unipotent, $X_l$ is $(n-l) \times (n-l)$ symmetric, $V_l$ is $l\times l$ positive diagonal, $Y_l$ is $(n-l)\times (n-l)$ positive definite symmetric, and $Q \in \mathrm{U}(n)$.
  
  We have
  \begin{align}
    \label{eq:cuspasymptotics}
    \Theta_f(h, g) = & (\det V_l)^{\frac{1}{4}}  (\det Y_l)^{\frac{1}{4}} \e\left( - t + \frac{1}{2} \bm{x}_l\transpose{\!\bm{y}_l}\right)  \nonumber \\
                     & \sum_{\bm{m}^{(2)} \in \mathbb{Z}^{n-l}} f_Q \left( \bm{x}_l^{(1)} V_l^{\frac{1}{2}}, (\bm{m}^{(2)} + \bm{x}_l^{(2)} )Y_l^{\frac{1}{2}} \right)  \nonumber \\
                     & \qquad\qquad\qquad \e \left( \frac{1}{2}( \bm{m}^{(2)} + \bm{x}_l^{(2)} ) X_l \transpose{\!(\bm{m}^{(2)} + \bm{x}_l^{(2)} )} + \bm{m}^{(2)} \transpose{\!\bm{y}_l}^{(2)} \right)  \nonumber \\
                     & + O_{A, f}\left( (\det V_l)^{\frac{1}{4}} (v_l + \bm{x}V \transpose{\!\bm{x}} )^{-A}\right),  
  \end{align}
  where
  \begin{equation}
    \label{eq:vjdef}
    V =
    \begin{pmatrix}
      v_1 & & \\
      & \ddots & \\
      & & v_n
    \end{pmatrix}
    ,
  \end{equation}
  and, with $
  \begin{pmatrix}
    \bm{x} & \bm{y} 
  \end{pmatrix}
  =
  \begin{pmatrix}
    \bm{x}^{(1)} & \bm{x}^{(2)} & \bm{y}^{(1)} & \bm{y}^{(2)}
  \end{pmatrix}
  $,
  \begin{multline}
    \label{eq:xjyjdef}
    \begin{pmatrix}
      \bm{x}_l^{(1)} & \bm{x}_l^{(2)} & \bm{y}_l^{(1)} & \bm{y}_l^{(2)}
    \end{pmatrix}
    \\
    =
    \begin{pmatrix}
      \bm{x}^{(1)} & \bm{x}^{(2)} & \bm{y}^{(1)} & \bm{y}^{(2)}
    \end{pmatrix}
    \begin{pmatrix}
      I & 0 & T_l & S_l \\
      0 & I & \transpose{\!S}_l & 0 \\
      0 & 0 & I & 0 \\
      0 & 0 & 0 & I
    \end{pmatrix}
    \begin{pmatrix}
      U_l & R_l & 0 & 0 \\
      0 & I & 0 & 0 \\
      0 & 0 & \transpose{\!U_l}^{-1} & 0 \\
      0 & 0 & - \transpose{\!R_l}\transpose{\!U_l}^{-1} & I
    \end{pmatrix}
    .
  \end{multline}
\end{theorem}

\begin{proof}
  Comparing the expressions
  \begin{equation}
    \label{eq:gXY1}
    g =
    \begin{pmatrix}
      I & X \\
      0 & I
    \end{pmatrix}
    \begin{pmatrix}
      Y^{\frac{1}{2}} & 0 \\
      0 & \transpose{Y}^{-\frac{1}{2}}
    \end{pmatrix}
    k(Q)
  \end{equation}
  and
  \begin{multline}
    \label{eq:gcoordinates}
    g = \begin{pmatrix}
      I & 0 & T_l & S_l \\
      0 & I & \transpose{\!S}_l & 0 \\
      0 & 0 & I & 0 \\
      0 & 0 & 0 & I
    \end{pmatrix}
    \begin{pmatrix}
      U_l & R_l & 0 & 0 \\
      0 & I & 0 & 0 \\
      0 & 0 & \transpose{\!U_l}^{-1} & 0 \\
      0 & 0 & - \transpose{\!R_l}\transpose{\!U_l}^{-1} & I
    \end{pmatrix}
    \\
    \begin{pmatrix}
      I & 0 & 0 & 0 \\
      0 & I & 0 & X_l \\
      0 & 0 & I & 0 \\
      0 & 0 & 0 & I
    \end{pmatrix}
    \begin{pmatrix}
      V_l^{\frac{1}{2}} & 0 & 0 & 0 \\
      0 & Y_l^{\frac{1}{2}} & 0 & 0 \\
      0 & 0 & V_l^{\frac{1}{2}} & 0 \\
      0 & 0 & 0 & \transpose{\!Y_l}^{-\frac{1}{2}}
    \end{pmatrix}
    k(Q),
  \end{multline}
  we find that
  \begin{equation}
    \label{eq:Xcoordinates}
    X =
    \begin{pmatrix}
      T_l + R_l X_l \transpose{\!R}_l & S_l + R_l X_l \\
      \transpose{\!S_l} + X_l \transpose{\!R}_l & X_l
    \end{pmatrix}
  \end{equation}
  and
  \begin{equation}
    \label{eq:Ycoordinates}
    Y^{\frac{1}{2}} =
    \begin{pmatrix}
      U_l V_l^{\frac{1}{2}} & R_l Y_l^{\frac{1}{2}} \\
      0 & Y_l^{\frac{1}{2}}
    \end{pmatrix}
    .
  \end{equation}
  Recalling from (\ref{eq:explicttheta}) that
  \begin{multline}
    \label{eq:explicttheta1}
    \Theta_f(h, g) = (\det Y)^{\frac{1}{4}} \e\left( - t + \frac{1}{2} \bm{x} \transpose{\!\bm{y}}\right) \\
    \sum_{\bm{m} \in \mathbb{Z}^n} f_Q\left( (\bm{m} + \bm{x})Y^{\frac{1}{2}}\right) \e\left( \frac{1}{2} (\bm{m} + \bm{x})X \transpose{\!(\bm{m} + \bm{x})} + \bm{m}\transpose{\!\bm{y}}\right),
  \end{multline}
  we express each term of the sum as
  \begin{align}
    \label{eq:termcoordinates}
    & f_Q\left((\bm{m} + \bm{x})Y^{\frac{1}{2}}\right) \e\left( \frac{1}{2} \bm{x}\transpose{\!\bm{y}} + \frac{1}{2} (\bm{m} + \bm{x})X \transpose{\!(\bm{m} + \bm{x})} + \bm{m}\transpose{\!\bm{y}}\right) \nonumber \\
    & = f_Q\left( (\bm{m}^{(1)} U_l + \bm{x}_l^{(1)}) V_l^{\frac{1}{2}}, (\bm{m}^{(1)} R_j + \bm{m}^{(2)} + \bm{x}_l^{(2)})Y_l^{\frac{1}{2}}\right) \nonumber \\
    & \qquad \e\bigg( \frac{1}{2} \bm{m}^{(1)} \left(T_l + R_l X_l \transpose{\!R_l}\right)\transpose{\!\bm{m}}^{(1)} \nonumber \\
    & \qquad\qquad + \bm{m}^{(1)} \transpose{\!\big( \bm{y}^{(1)} + \bm{x}^{(1)}( T_l + R_lX_l\transpose{\!R_l}) + (\bm{m}^{(2)} + \bm{x}^{(2)}) (\transpose{\!S_l} + X_l \transpose{\!R_l})\big)}\bigg) \nonumber \\
    & \qquad \e\bigg(\frac{1}{2}\bm{x}_l^{(1)} \transpose{\!\bm{y}_l}^{(1)} + \frac{1}{2} \bm{x}_l^{(2)} \transpose{\!\bm{y}_l}^{(2)} \nonumber \\
    & \qquad\qquad + \frac{1}{2} ( \bm{m}^{(2)} + \bm{x}^{(2)}) X_l \transpose{\!(\bm{m}^{(2)} + \bm{x}^{(2)})} + \bm{m}^{(2)} \transpose{\!\bm{y}_l}^{(2)} \bigg),
  \end{align}
  where $\bm{m} =
  \begin{pmatrix}
    \bm{m}^{(1)} & \bm{m}^{(2)}
  \end{pmatrix}
  $, and $
  \begin{pmatrix}
    \bm{x}_l & \bm{y}_l
  \end{pmatrix}
  =
  \begin{pmatrix}
    \bm{x}_l^{(1)} & \bm{x}_l^{(2)} & \bm{y}_l^{(1)} & \bm{y}_l^{(2)}
  \end{pmatrix}
  $ is given by (\ref{eq:xjyjdef}).
  We observe from (\ref{eq:termcoordinates}) that the main term in (\ref{eq:cuspasymptotics}) is the sum over those $\bm{m} =
  \begin{pmatrix}
    \bm{m}^{(1)} & \bm{m}^{(2)}
  \end{pmatrix}
  $ with $\bm{m}^{(1)} = 0$.

  To bound the contribution of the terms with $\bm{m}^{(1)} \neq 0$, we proceed by induction on $l$, making use of the recursive definition of the fundamental domain $\mathcal{D}'_n$ containing $Y$. 
  For $l =1$, the contribution of $m^{(1)} \neq 0$ is, by lemma \ref{lemma:uniformSchwartz},
  \begin{align}
    \label{eq:j1m1not0}
    & \ll (\det Y)^{\frac{1}{4}} \sum_{\substack{m^{(1)} \in \mathbb{Z} \\ m^{(1)} \neq 0}} \sum_{\bm{m}^{(2)} \in \mathbb{Z}^{n-1}} \left( 1 + (\bm{m} + \bm{x}) Y \transpose{\!(\bm{m} + \bm{x})}\right)^{-A} \nonumber \\
    & \quad \ll v_1^{\frac{1}{4}} (\det Y_1)^{\frac{1}{4}}  \sum_{\substack{m^{(1)} \in \mathbb{Z} \\ m^{(1)} \neq 0}} \sum_{\bm{m}^{(2)} \in \mathbb{Z}^{n-1}} \nonumber \\
    & \qquad \left(1 + v_1( m^{(1)} + x^{(1)})^2 + ( \bm{m}^{(2)} + \bm{x}^{(2)}) Y_1 \transpose{\!(\bm{m}^{(2)} + \bm{x}^{(2)})}\right)^{-A}
  \end{align}
  by proposition \ref{proposition:Dnproperties}.
  Applying (\ref{eq:integersum1}) with $v = v_1$, $x = x^{(1)}$, $y = 1 + ( \bm{m}^{(2)} + \bm{x}^{(2)}) Y_1 \transpose{\!(\bm{m}^{(2)} + \bm{x}^{(2)})}$, and renaming $\bm{m}^{(2)} = \bm{m}_1$, $\bm{x}^{(2)} = \bm{x}_1$ (not to be confused with (\ref{eq:xjyjdef})), this is
  \begin{align}
    \label{eq:j1m1not02}
    & \ll v_1^{-\frac{1}{4}} (\det Y_1)^{\frac{1}{4}} \sum_{\bm{m}_1 \in \mathbb{Z}^{n-1}} \left(v_1 + (\bm{m}_1 + \bm{x}_1) Y_1 \transpose{\!(\bm{m}_1 + \bm{x}_1)}\right)^{-A} \nonumber \\
    & \quad \ll v_1^{- \frac{1}{4}} v_2^{\frac{1}{4}} ( \det Y_2)^{\frac{1}{4}} \sum_{m_1^{(1)} \in \mathbb{Z}} \sum_{\bm{m}_1^{(2)} \in \mathbb{Z}^{n-2}} \nonumber \\
    & \qquad \left(v_1 + v_2(m_1^{(1)} + x_1^{(1)})^2 + (\bm{m}_1^{(2)} + \bm{x}_1^{(2)}) Y_1 \transpose{\!(\bm{m}_1^{(2)} + \bm{x}_1^{(2)})}\right)^{-A}
  \end{align}
  by proposition \ref{proposition:Dnproperties}, recalling that $Y_1 \in \mathcal{D}'_{n-1}$.
  Applying (\ref{eq:integersum2}) with $v = v_2$, $x = x_1^{(1)}$, $y = v_1 + (\bm{m}_1^{(2)} + \bm{x}_1^{(2)}) Y_1 \transpose{\!(\bm{m}_1^{(2)} + \bm{x}_1^{(2)})} \gg v_2$, this is
  \begin{equation}
    \label{eq:n1m1not03}
    \ll v_1^{-\frac{1}{4}} v_2^{-\frac{1}{4}} (\det Y_2)^{\frac{1}{4}} \sum_{\bm{m}_1^{(2)} \in \mathbb{Z}^{n-2}} \left( v_1 + (\bm{m}_1^{(2)} + \bm{x}_1^{(2)}) Y_2 \transpose{\!(\bm{m}_1^{(2)} + \bm{x}_1^{(2)})}\right)^{-A}. 
  \end{equation}
  Continuing in this way, we eventually obtain the bound
  \begin{equation}
    \label{eq:j1final}
    \ll v_1^{-\frac{1}{4}} \cdots v_n^{-\frac{1}{4}} (v_1)^{-A} \ll  v_1^{\frac{1}{4}}( v_1 + \bm{x} V \transpose{\!\bm{x}})^{-A},
  \end{equation}
  thus establishing (\ref{eq:cuspasymptotics}) for $l =1$.

  For $l > 1$, we see by induction, lemma \ref{lemma:uniformSchwartz}, and proposition \ref{proposition:Dnproperties} that we need to bound
  \begin{multline}
    \label{eq:j2m1not0}
    (\det Y)^{\frac{1}{4}} \sum_{\substack{\bm{m}^{(1)} = (0\ m) \in \mathbb{Z}^{l} \\ m \neq 0}} \sum_{\bm{m}^{(2)} \in \mathbb{Z}^{n-l}} \\
    \left( (\bm{m}^{(1)} + \bm{x}^{(1)}) V_l \transpose{\! (\bm{m}^{(1)} + \bm{x}^{(1)})} + (\bm{m}^{(2)} + \bm{x}^{(2)}) Y_l \transpose{\! (\bm{m}^{(2)} + \bm{x}^{(2)})} \right)^{-A}. 
  \end{multline}
  Applying (\ref{eq:integersum1}) with $v = v_l$, $x$ the last entry of $\bm{x}^{(1)}$, $y = \bm{x}^{(1)} V_l \transpose{\!\bm{x}^{(1)}} - v_lx^2$, and renaming $\bm{m}^{(2)} = \bm{m}_1$, $\bm{x}^{(2)} = \bm{x}_1$, this is
  \begin{align}
    \label{eq:j2m1not02}
    & \ll v_l^{-\frac{1}{2}}(\det V_l)^{\frac{1}{4}} ( \det Y_l)^{\frac{1}{4}} \nonumber \\
    & \quad \sum_{\bm{m}_1 \in \mathbb{Z}^{n-l}} \left( v_l + \bm{x}^{(1)} V_l \transpose{\!\bm{x}^{(1)}} + (\bm{x}_1 + \bm{m}_1) Y_l \transpose{\!(\bm{x}_1 + \bm{m}_1)}\right)^{-A} \nonumber \\ 
    & \ll v_l^{-\frac{1}{2}} (\det V_{l})^{\frac{1}{4}} ( \det Y_l)^{\frac{1}{4}} \nonumber \\
    & \quad \sum_{m_1^{(1)} \in \mathbb{Z}} \sum_{\bm{m}_1^{(2)} \in \mathbb{Z}^{n-l-1}} \Big(v_l + \bm{x}^{(1)} V_l \transpose{\!\bm{x}^{(1)}} + \nonumber \\
    & \qquad v_{l+1} (m_1^{(1)} + x_1^{(1)})^2 + (\bm{m}_1^{(2)} + \bm{x}_1^{(2)}) Y_{l+1} \transpose{\!(\bm{m}_1^{(2)} + \bm{x}_1^{(2)})} \Big)^{-A}
  \end{align}
  by proposition \ref{proposition:Dnproperties} and $Y_l \in \mathcal{D}_{n-l}'$. 
  Applying (\ref{eq:integersum2}) repeatedly as we did in the $l =1$ case, we obtain the bound
  \begin{equation}
    \label{eq:j2m1not03}
    \ll (\det V_l)^{\frac{1}{4}} v_l^{-\frac{1}{2}} v_{l+1}^{-\frac{1}{4}} \cdots v_n^{-\frac{1}{4}} ( v_l + \bm{x}^{(1)} V_l \transpose{\!\bm{x}^{(1)}})^{-A} \\
    \ll (\det V_l)^{\frac{1}{4}} ( v_l + \bm{x} V \transpose{\!\bm{x}})^{-A}
  \end{equation}
  as required. 
\end{proof}

Since $v_n \geq \frac{\sqrt{3}}{2}$ for $g \in \mathcal{D}$ by proposition \ref{proposition:Dnproperties}, we obtain the following corollary. 

\begin{corollary}
  \label{corollary:upperbound}
  For a Schwartz function $f\in\mathcal{S}(\mathbb{R}^n)$, $g \in \mathcal{D} $, and $h = (\bm{x}, \bm{y}, t) \in H$ with the entries of $\bm{x}$ and $\bm{y}$ at most $\frac{1}{2}$ in absolute value, we have
  \begin{equation}
    \label{eq:thetadetbound}
    \Theta_f(h, g) \ll_f (\det V)^{\frac{1}{4}}( 1 + \bm{x} V \transpose{\!\bm{x}})^{-A}
  \end{equation}
 where
  \begin{equation}
    \label{eq:gXYkQ1}
    g =
    \begin{pmatrix}
      I & X \\
      0 & I 
    \end{pmatrix}
    \begin{pmatrix}
      Y^{\frac{1}{2}} & 0 \\
      0 & \transpose{Y}^{-\frac{1}{2}}
    \end{pmatrix}
    k(Q)
  \end{equation}
  with $Y = U V \transpose{U}$ as usual. 
  \end{corollary}

\section{Proof of the main theorems}
\label{sec:mainproof}

Having the bounds from corollary \ref{corollary:upperbound}, we now proceed to the proof of theorems \ref{theorem:thetasumboundssmooth} and \ref{theorem:thetasumbounds}.
In the smooth setting of theorem \ref{theorem:thetasumboundssmooth}, we need to construct a distance-like (DL) function that captures the bounds in corollary \ref{corollary:upperbound}. This will enable us to directly apply theorem 1.7 in \cite{KleinbockMargulis1999} modulo a standard argument that allows us to pass from a full measure set in $\Gamma g \in \Gamma \backslash G$ to a full measure set on the unstable foliation parametrized by $X \in \mathbb{R}^{n\times n}_{\mathrm{sym}}$.
The proof of theorem \ref{theorem:thetasumbounds} is more involved, and requires modifications of the method in \cite{KleinbockMargulis1999} to enable a resolution of the singular cutoff function in (\ref{eq:SMXxydef}). To this end we need to uniformly manage many points in $\Gamma \backslash G$. 

We note that theorem 1.7 in \cite{KleinbockMargulis1999}  is also a main input in the method of \cite{ConsentinoFlaminio2015}.

\subsection{Heights and volumes}
\label{sec:heightsvolumes}

We define the height function $D: \Gamma \backslash G \to \mathbb{R}_{> 0}$ by
\begin{equation}
  \label{eq:Ddef}
  D\left( \Gamma g \right) = \max_{\gamma \in \Gamma} \det V(\gamma g) = \det V(\gamma_0 g)
\end{equation}
where $\gamma_0$ is such that $\gamma_0 g \in \mathcal{D}$ and we write
\begin{equation}
  \label{eq:gXYkQ2}
  g =
  \begin{pmatrix}
    U & X\transpose{U}^{-1} \\
    0 & \transpose{U}^{-1} 
  \end{pmatrix}
  \begin{pmatrix}
    V^{\frac{1}{2}} & 0 \\
    0 & V^{-\frac{1}{2}}
  \end{pmatrix}
  k(Q)
\end{equation}
with $ V = V(g)$ positive diagonal as usual.
We remark that from corollary \ref{corollary:upperbound} and the automorphy of $\Theta_f$, theorem \ref{theorem:thetaautomorphy}, we have $\Theta_f(h, g) \ll D(\Gamma g)^{\frac{1}{4}}$ for all $(h, g) \in H \rtimes G$ with the implied constant depending only on $f$.
We also remark that the logarithm of $D$ is a distance-like function in the sense of \cite{KleinbockMargulis1999}, see also \cite{ConsentinoFlaminio2015}. 

We begin by estimating the measure of the set on which $D$ is large, thus verifying one of the required properties for the logarithm of $D$ to be a $\frac{n+1}{2}$-DL function, see \cite{KleinbockMargulis1999}.
This estimate is also found in \cite{ConsentinoFlaminio2015} and the relevant change of variables in \cite{Klingen1990}. 

\begin{lemma}
  \label{lemma:volumecalculation}
  Let $\mu$ be Haar measure on $G$ and $R > 0$.
  We have
  \begin{equation}
    \label{eq:volumebound}
    \mu( \{ \Gamma g \in  \Gamma \backslash G : D(\Gamma g) \geq R \}) \ll R^{- \frac{n+1}{2}}
  \end{equation}
  with the implied constant depending only on $n$.
\end{lemma}

\begin{proof}
  We recall that $g \in \mathcal{D}$ is written as
  \begin{equation}
    \label{eq:gUVXQ}
    g =
    \begin{pmatrix}
      U & X\transpose{U}^{-1} \\
      0 & \transpose{U}^{-1} 
    \end{pmatrix}
    \begin{pmatrix}
      V^{\frac{1}{2}} & 0 \\
      0 & V^{-\frac{1}{2}}
    \end{pmatrix}
    k(Q)
  \end{equation}
  for $U$ upper-triangular unipotent, $X$ symmetric, $Q \in \mathrm{U}(n)$, and
  \begin{equation}
    \label{eq:Vrecal}
    V = V(g) =
    \begin{pmatrix}
      v_1 & \cdots & 0 \\
      \vdots & \ddots & \vdots \\
      0 & \cdots & v_n
    \end{pmatrix}
  \end{equation}
  positive diagonal.
  The Haar measure $\mu$ is then Lebesgue measure with respect to the entries of $X$ and the off-diagonal entries of $U$, $\mathrm{U}(n)$-Haar measure on $Q$, and the measure given by
  \begin{equation}
    \label{eq:VHaarmeasure}
    v_1^{-n-1} v_2^{-n} \cdots v_n^{-2} \dd v_1 \dd v_2 \cdots \dd v_n
  \end{equation}
  on $V$.

  From the construction of $\mathcal{D}$, it is clear that the entries of $U$ and $X$ are constrained to a compact region.
  Since $\mathrm{U}(n)$ is also compact, we have by proposition \ref{proposition:Dnproperties} that
  \begin{equation}
    \label{eq:volumebound2}
    \mu( \{ g \in \mathcal{D}_n : \det V(g) \geq R \}) \\
    \ll \underset{ \substack{ v_j \geq \frac{3}{4} v_{j+1} \\ v_1 \cdots v_n  \geq R}}{\int \cdots \int} v_1^{-n-1} v_2^{-n} \cdots v_n^{-2} \dd v_1 \dd v_2 \cdots \dd v_n.
  \end{equation}


  Changing variables $v_j = \exp(u_j)$, the integral in (\ref{eq:volumebound2}) is
  \begin{equation}
    \label{eq:ujintegral}
    \underset{ \substack{ u_j - u_{j+1} \geq \log \frac{3}{4} \\ u_1 + \cdots + u_n \geq \log R}}{\int \cdots \int} \exp( -nu_1 - (n-1)u_2 - \cdots - u_n) \dd u_1 \dd u_2 \cdots \dd u_n. 
  \end{equation}
  We now make the linear change of variables $s_j = u_j - u_{j+1}$ for $ j < n$ and $s_n = u_1 + \cdots + u_n$.
  This transformation has determinant $n$ and its inverse is given by
  \begin{equation}
    \label{eq:ujequals}
    u_j = - \frac{1}{n} \sum_{1\leq i < j} i s_i + \frac{1}{n} \sum_{j\leq i < n} (n-i) s_i + \frac{1}{n} s_n. 
  \end{equation}
  We find that the exponent in (\ref{eq:ujintegral}) is then
  \begin{equation}
    \label{eq:exponent}
    - \sum_{1\leq j \leq n} (n - j + 1) u_j = - \frac{n+1}{2} s_n - \sum_{1 \leq j < n} \frac{j(n-j)}{2} s_j. 
  \end{equation}
  As $ \frac{j(n-j)}{2} > 0$ for $j < n$, the bound (\ref{eq:volumebound}) follows.
\end{proof}

We now control the change in the height function $D$ under a geodesic flow by a fixed distance.
This estimate should be compared to the requirement in \cite{KleinbockMargulis1999} that distance-like functions be uniformly continuous.

\begin{lemma}
  \label{lemma:heightcontinuity}
  For $\Gamma g \in \Gamma \backslash G$ and $|s| \leq 1$, we have
  \begin{equation}
    \label{eq:heightcontinuity}
    D\bigg( \Gamma g
    \begin{pmatrix}
      \e^{s} I & 0 \\
      0 & \e^{-s} I 
    \end{pmatrix}
    \bigg) \asymp D(\Gamma g)
  \end{equation}
  with implied constants depending only on $n$. 
\end{lemma}

\begin{proof}
  For arbitrary $g \in G$ and $|s|\leq 1$, we set
  \begin{equation}
    \label{eq:gsdef}
    g_s = g
    \begin{pmatrix}
      \e^s I & 0 \\
      0 & \e^{-s} I
    \end{pmatrix}
    .
  \end{equation}
  We first claim that
    \begin{equation}
    \label{eq:detcontinuity}
    \det V(g_s) \asymp \det  V(g)
  \end{equation}
  for all $g \in G$.

  As usual we have
  \begin{equation}
    \label{eq:gXYkQ}
    g =
    \begin{pmatrix}
      I & X \\
      0 & I 
    \end{pmatrix}
    \begin{pmatrix}
      Y^{\frac{1}{2}} & 0 \\
      0 & \transpose{Y}^{-\frac{1}{2}}
    \end{pmatrix}
    k(Q),
  \end{equation}
  and we note that $\det Y = \det V(g)$.
  Writing
  \begin{equation}
    \label{eq:QRealImaginary}
    Q = R + \mathrm{i} S,
  \end{equation}
  we have
  \begin{equation}
    \label{eq:gflow}
    g_s
    =
    \begin{pmatrix}
      I & X \\
      0 & I
    \end{pmatrix}
    \begin{pmatrix}
      \e^s Y^{\frac{1}{2}} R & - \e^{-s} Y^{\frac{1}{2}} S \\
      \e^s \transpose{Y}^{-\frac{1}{2}} S & \e^{-s} \transpose{Y}^{-\frac{1}{2}} R
    \end{pmatrix}
    ,
  \end{equation}
  so in view of (\ref{eq:XYQ})
  \begin{equation}
    \label{eq:detflow}
    Y(g_s) = Y^{\frac{1}{2}} \left(\e^{2s}  S\transpose{\!S} + \e^{-2s} R \transpose{\!R}  \right)^{-1} \transpose{Y}^{\frac{1}{2}}.
  \end{equation}
  The ratio of the right to the left side of (\ref{eq:detcontinuity}) is then
  \begin{equation}
    \label{eq:detratio}
    \det ( \e^{2s} S \transpose{\!S} + \e^{-2s} R \transpose{\!R}).
  \end{equation}

  Using the diagonalization argument from the proof of lemma \ref{lemma:uniformSchwartz}, we can multiply by orthogonal matrices to make $R$ and $S$ diagonal with entries $\cos \phi_j$ and $\sin \phi_j$.
  The determinant (\ref{eq:detratio}) is then
  \begin{equation}
    \label{eq:detdiagonal}
    \prod_{1\leq j \leq n} ( \e^{2s} \sin^2 \phi_j + \e^{-2s} \cos^2 \phi_j).
  \end{equation}
  Since $|s|\leq 1$, this is clearly bounded from above by a constant depending on $n$, and since $\sin^2 \phi_j$ and $\cos^2 \phi_j$ cannot both be less than $\frac{1}{2}$, it is also bounded away from $0$.
  This establishes (\ref{eq:detcontinuity}).

  Now we have
  \begin{equation}
    \label{eq:maxgamma}
    D(\Gamma g_s) = \max_{\gamma \in \Gamma} \det V(\gamma g_s) = \det V (\gamma_0 g_s)
  \end{equation}
  for some $\gamma_0 \in \Gamma$. 
  By (\ref{eq:detcontinuity}) we have $D(\Gamma g_s) \ll \det V(\gamma_0 g) \leq D(\Gamma g)$. 
  The same reasoning with $g_s$ replaced by $g$ leads to the reverse inequality, establishing (\ref{eq:heightcontinuity}). 
\end{proof}

The following lemma is similar to lemma \ref{lemma:heightcontinuity} in that we control the change in $D$ under a particular action.
Here the action is more general, however we only need to consider small neighbourhoods in $\Gamma \backslash G$.

\begin{lemma}
  \label{lemma:heightcontinuity2}
  There exists a constant $\epsilon_n > 0$ depending only on $n$ such that for all $\Gamma g \in \Gamma \backslash G$, $A \in \mathrm{GL}(n, \mathbb{R})$ satisfying $|| A - I || \leq \epsilon_n$, and symmetric $T$ satisfying $|| T || \leq \epsilon_n$, we have
  \begin{equation}
    \label{eq:heightcontinuity2}
    D \bigg( \Gamma g
    \begin{pmatrix}
      A & 0 \\
      0 & \transpose{\!A}^{-1}
    \end{pmatrix}
    \begin{pmatrix}
      I & 0 \\
      T & I
    \end{pmatrix}
    \bigg) \asymp D( \Gamma g). 
  \end{equation}
\end{lemma}

\begin{proof}
  As in the proof of lemma \ref{lemma:heightcontinuity}, it suffices to show that for all $g \in G$,
  \begin{equation}
    \label{eq:detcontinuity2}
    \det V(gg_Ag_T) \asymp \det V(g),
  \end{equation}
  where
  \begin{equation}
    \label{eq:gXdef}
    g_A =
    \begin{pmatrix}
      A & 0 \\
      0 & \transpose{\!A}^{-1}
    \end{pmatrix}
    ,\quad g_T = g
    \begin{pmatrix}
      I & 0 \\
      T & I
    \end{pmatrix}
    .
  \end{equation}
  We write
  \begin{equation}
    \label{eq:iwasawa}
    g =
    \begin{pmatrix}
      I & X \\
      0 & I
    \end{pmatrix}
    \begin{pmatrix}
      Y^{\frac{1}{2}} & 0 \\
      0 & \transpose{Y}^{-\frac{1}{2}}
    \end{pmatrix}
    \begin{pmatrix}
      R & -S \\
      S & R
    \end{pmatrix}
    ,
  \end{equation}
  where $R + \mathrm{i}S \in \mathrm{U}(n)$.
  We compute
  \begin{equation}
    \label{eq:YgTinverse}
    Y(g g_A g_T)^{-1} \\
    = \transpose{Y}^{-\frac{1}{2}} \left( ( SA + R\transpose{\!A}^{-1}T) \transpose{\!( SA + R\transpose{\!A}^{-1}T)} + R\transpose{\!A}^{-1} A^{-1} \transpose{\!R} \right) Y^{-\frac{1}{2}},
  \end{equation}
  so the ratio of the left and right sides of (\ref{eq:detcontinuity2}) is
  \begin{align}
    \label{eq:detratio1}
    & \det(  ( SA + R\transpose{\!A}^{-1}T) \transpose{\!( SA + R\transpose{\!A}^{-1}T)} + R\transpose{\!A}^{-1} A^{-1} \transpose{\!R} ) \nonumber \\
    & \quad = \det \big( SA \transpose{\!A}\transpose{\!S} + R\transpose{\!A}^{-1} A^{-1} \transpose{\!R} \nonumber \\
    & \qquad\qquad + R \transpose{\!A}^{-1} T \transpose{\!A} \transpose{\!S} + S A T A^{-1} \transpose{\!R} + R \transpose{\!A}^{-1}  T^2 A^{-1} \transpose{\!R} \big).
  \end{align}

  Recalling that $R\transpose{\!R} + S\transpose{\!S} = I$, we have
  \begin{equation}
    \label{eq:AnearI}
    SA \transpose{\!A}\transpose{\!S} + R\transpose{\!A}^{-1} A^{-1} \transpose{\!R} = I + O ( \epsilon_n)
  \end{equation}
  if $|| A - I || \leq \epsilon_n$.
  It follows that $\epsilon_n$ can be made sufficiently small so that the symmetric matrix
  \begin{equation}
    \label{eq:Iminusmatrix}
    - I + SA \transpose{\!A}\transpose{\!S} + R\transpose{\!A}^{-1} A^{-1} \transpose{\!R} \\
    + R \transpose{\!A}^{-1} T \transpose{\!A} \transpose{\!S} + S A T A^{-1} \transpose{\!R} + R \transpose{\!A}^{-1}  T^2 A^{-1} \transpose{\!R} 
  \end{equation}
  has all eigenvalues less than $\frac{1}{n}$, say, in absolute value, and (\ref{eq:detcontinuity2}) follows. 
\end{proof}

\subsection{Proof of theorem \ref{theorem:thetasumboundssmooth}}
\label{sec:bound00}

In this section we sketch a proof of theorem \ref{theorem:thetasumboundssmooth}, appealing to the method in \cite{KleinbockMargulis1999}.
A complete proof of theorem \ref{theorem:thetasumboundssmooth} can be obtained from the proof of theorem \ref{theorem:thetasumbounds} in the following section by only considering the $\bm{j}=(0,\ldots,0)$ term in the dyadic expansion, for a general Schwartz function $f$ rather than the compactly supported function $f_n$ considered.

We have
\begin{equation}
  \label{eq:thetafThetaf}
  \theta_f(M, X, \bm{x}, \bm{y}) = M^{\frac{n}{2}} \Theta_f(h, g_{M, X})
\end{equation}
where $h = (\bm{x}, \bm{y}, 0)$ and
\begin{equation}
  \label{eq:gMX}
  g_{M, X}  =
  \begin{pmatrix}
    I & X \\
    0 & I
  \end{pmatrix}
  \begin{pmatrix}
    M^{-1} I & 0 \\
    0 & M I
  \end{pmatrix}
  .
\end{equation}
By corollary \ref{corollary:upperbound} we have
\begin{equation}
  \label{eq:thetabound0}
  \theta_f(M, X, \bm{x}, \bm{y}) \ll_f M^{\frac{n}{2}}D(\Gamma g_{M, X})^{\frac{1}{4}}.
\end{equation}

Now for the proof of the easy part of theorem 1.7 in \cite{KleinbockMargulis1999}, we observe that one only needs the upper bound in lemma \ref{lemma:volumecalculation} instead of the matching lower bound in the definition of $\frac{n+1}{2}$-DL functions.
In addition, one does not need that the function $\log D(\Gamma g)$ be uniformly continuous; lemma \ref{lemma:heightcontinuity} suffices.
We therefore have that
\begin{equation}
  \label{eq:Dbound}
  D\left(\Gamma g
  \begin{pmatrix}
    M^{-1} I & 0 \\
    0 & M I 
  \end{pmatrix}
  \right)^{\frac{1}{4}} \ll_g  \psi(\log M)
\end{equation}
for almost all $\Gamma g \in \Gamma \backslash G$ as long as
\begin{equation}
  \label{eq:psiconverge}
  \sum_{k\geq 0} \psi(k)^{-2n -2} < \infty
\end{equation}
with $ \psi: [0, \infty) \to [1, \infty)$ increasing.

To finish our proof sketch, we consider the set of $X \in \mathbb{R}^{n\times n}_{\mathrm{sym}}$ such that there exist $A\in \mathrm{GL}(n, \mathbb{R})$ with $|| A - I|| \leq \epsilon_n$ and $T \in \mathbb{R}^{n\times n}_{\mathrm{sym}}$ with $|| T|| \leq \epsilon_n$ so that
\begin{equation}
  \label{eq:nearbyg}
  \Gamma g = \Gamma
  \begin{pmatrix}
    I & X \\
    0 & I
  \end{pmatrix}
  \begin{pmatrix}
    A & 0 \\
    0 & \transpose{A}^{-1}
  \end{pmatrix}
  \begin{pmatrix}
    I & 0 \\
    T & I
  \end{pmatrix}
\end{equation}
satisfies the bound (\ref{eq:Dbound}).
From the Haar measure calculation (\ref{eq:XATHaar}) we see that this set of $X$ has full measure, and from lemma \ref{lemma:heightcontinuity2} and (\ref{eq:Dbound}) we see that
\begin{equation}
  \label{eq:Dbound1}
  D(\Gamma g_{M, X})^{\frac{1}{4}} \ll_X \psi(\log M)
\end{equation}
for all $X$ in this set.
Theorem \ref{theorem:thetasumboundssmooth} then follows. 


\subsection{Proof of theorem \ref{theorem:thetasumbounds}}
\label{sec:bound}


We record the following lemma that dyadically decomposes the indicator function of the open interval $(0, 1)$, noting that both the singularities at $0$ and $1$ need to be resolved.

\begin{lemma}
  \label{lemma:cutoff}
  There exists a smooth, compactly supported function $f_1 : (0,1) \to \mathbb{R}_{\geq 0}$ such that
  \begin{equation}
    \label{eq:chidecomp}
    \chi_1(x) = \sum_{j \geq 0} \left( f_1\left( 2^j x\right) +  f_1\left( 2^j( 1 -x)\right)\right),
  \end{equation}
  where $\chi_1$ is the indicator function of the open unit interval $(0,1)$. 
\end{lemma}

We note that \cite{CellarosiMarklof2016} has an explicit construction of a function $f_1$ satisfying (\ref{eq:chidecomp}) that is however only twice differentiable.
As we make no effort here to determine constants in our estimates, we sacrifice explicitness for smoothness. 

\begin{proof}
  We let $f_0$ be a non-negative, smooth function such that $f_0(x) = 0$ for $x \leq 0$, $f_0(x) = 1$ for $x \geq 1$, and
  \begin{equation}
    \label{eq:f0condtion}
    f_0(x) + f_0(1 -x) = 1
  \end{equation}
  for all $0\leq x \leq 1$.
  We then define the smooth function $f_1$ on the interval by
  \begin{equation}
    \label{eq:f1def}
    f_1(x) =
    \begin{cases}
      0 & \mathrm{if\ } x \leq \frac{1}{6} \\
      f_0( 6x - 1) & \mathrm{if\ } \frac{1}{6} \leq x \leq \frac{1}{3} \\
      f_0( 2 - 3x) & \mathrm{if\ } \frac{1}{3} \leq x \leq \frac{2}{3} \\
      0 & \mathrm{if\ } \frac{2}{3} \leq x. 
    \end{cases}
  \end{equation}

  Let us now consider the expression
  \begin{equation}
    \label{eq:f1sum}
    \sum_{j \geq 0} \left( f_1(2^j x) + f_1( 2^j(1 -x)) \right),
  \end{equation}
  which is clearly $0$ if $x \not\in (0,1)$.
  If $ 0 < x \leq \frac{1}{3}$ then $f_1(2^j(1 -x)) = 0$ for all $j \geq 0$, and $f_1(2^jx)$ is nonzero for exactly two values of $j\geq 0$, say $j_0$ and $j_0 +1$.
  We have
  \begin{equation}
    \label{eq:j0term}
    f_1(2^{j_0} x) + f_1(2^{j_0 + 1} x) = f_0( 2^{j_0} 6x - 1) + f_0( 2 - 2^{j_0} 6 x) = 1
  \end{equation}
  by (\ref{eq:f0condtion}).
  We similarly find that (\ref{eq:f1sum}) is $1$ for $\frac{2}{3} \leq x < 1$.
  When $\frac{1}{3} \leq x \leq \frac{2}{3}$, only the $j = 0$ term in (\ref{eq:f1sum}) is nonzero.
  We have
  \begin{equation}
    \label{eq:f1middle}
    f_1(x) + f_1(1 -x) = f_0( 2 - 3x) + f_0( 3x - 1) = 1
  \end{equation}
  by the condition (\ref{eq:f0condtion}). 
\end{proof}

For a subset $S \subset \{ 1, \dots, n\}$ and $\bm{j} = (j_1, \dots, j_n) \in \mathbb{Z}^n$ with $j_i \geq 0$, we define
\begin{equation}
  \label{eq:gjdef}
  g_{\bm{j}, S} =
  \begin{pmatrix}
    A_{\bm{j}}E_S & 0 \\
    0 & A_{\bm{j}}^{-1}E_S
  \end{pmatrix}
  \in G,
\end{equation}
where $E_S$ is diagonal with $(i,i)$ entry $-1$ if $i \in S$, $+1$ if $i \not\in S$ and 
\begin{equation}
  \label{eq:Ajdef}
  A_{\bm{j}} =
  \begin{pmatrix}
    2^{j_1} & \cdots & 0 \\
    \vdots & \ddots & \vdots \\
    0 & \cdots & 2^{j_n}
  \end{pmatrix}
  .
\end{equation}
We also set
\begin{equation}
  \label{eq:hSdef}
  h_S = ( \bm{x}_S, 0, 0) \in H
\end{equation}
where $\bm{x}_S$ has $i$th entry $-1$ if $i \in S$ and $0$ if $i \not\in S$.

We observe that from lemma \ref{lemma:cutoff} we have
\begin{equation}
  \label{eq:chidecomp1}
  \chi(\bm{x}) = \sum_{\bm{j} \geq 0} \sum_S f_n\left( ( \bm{x} + \bm{x}_S) E_S A_{\bm{j}} \right),
\end{equation}
where $\chi$ is the indicator function of the open unit cube $(0,1)^n$,
\begin{equation}
  \label{eq:fdef}
  f_n(x_1, \dots, x_n) = \prod_{1\leq j \leq n} f_1(x_j),
\end{equation}
and the sums are over $\bm{j} \in \mathbb{Z}^n$ with nonnegative entries and all subsets $S$ of $\{1, \dots, n\}$.
The characteristic function of the rectangular box $\mathcal{B}=(0,b_1)\times\cdots\times(0,b_n)$ is therefore
\begin{equation}
  \label{eq:chidecomp1bbb}
  \chi_\mathcal{B}(\bm{x}) = \chi(\bm{x} B^{-1}) = \sum_{\bm{j} \geq 0} \sum_S f_n\left( ( \bm{x}B^{-1} + \bm{x}_S) E_S A_{\bm{j}}  \right),
\end{equation}
where $B$ is the diagonal matrix with coefficients $b_1,\ldots,b_n$.

Recalling the Schr\"odinger representation $W$ and the Segal-Shale-Weil representation $R$, we have
\begin{equation}
  \label{eq:chirepdecomp}
  \chi(\bm{x}) = \sum_{\bm{j} \geq 0} \sum_{S} 2^{-\frac{1}{2}( j_1 + \cdots + j_n)} \left(W(h_S) R(g_{\bm{j}, S}) f_n\right) ( \bm{x})
\end{equation}
and
\begin{equation}
  \label{eq:chirepdecompbbb}
  \begin{split}
  \chi_\mathcal{B}(\bm{x}) & = (\det B)^{\frac12} (R(\begin{pmatrix} B^{-1} & 0 \\ 0 & B\end{pmatrix} ) \chi)(\bm{x})  
  \\ 
  & = (\det B)^{\frac12}  \sum_{\bm{j} \geq 0} \sum_{S} 2^{-\frac{1}{2}( j_1 + \cdots + j_n)} \left(R(\begin{pmatrix} B^{-1} & 0 \\ 0 & B\end{pmatrix}) W(h_S) R(g_{\bm{j}, S}) f_n\right) ( \bm{x}) .
  \end{split}
\end{equation}

We note that for $(h, g) \in H\rtimes G$,
\begin{equation}
  \label{eq:reprearrange}
  W(h) R(g) W(h_S) R(g_{\bm{j}, S}) = W( h h_S^{g^{-1}}) R(g g_{\bm{j}, S}),
\end{equation}
and so, as we are interested in the theta sums (\ref{eq:SMXxydef}) with the sharp cutoff $\chi$, it is natural to consider the expression
\begin{equation}
  \label{eq:thetachidef}
  \tilde{\Theta}_\chi (h, g) = \sum_{\bm{j}\geq 0} \sum_S 2^{-\frac{1}{2}( j_1 + \cdots + j_n)} \Theta_{f_n}\left( h h_S^{g^{-1}}, g g_{\bm{j}, S}\right).
\end{equation}
The convergence of this expression for almost every $g \in G$ is a corollary of lemma \ref{lemma:thetajbound} below.

Motivated by bounding (\ref{eq:thetachidef}) via corollary \ref{corollary:upperbound}, for $C > 0$ and $\psi : [0,\infty) \to [1, \infty)$ an increasing function, we define $\mathcal{G}_{\bm{j}}(\psi, C)$ to be the set of $\Gamma g \in \Gamma \backslash G$ such that
\begin{equation}
  \label{eq:GpsiCdef}
  D\bigg( \Gamma g g_{\bm{j}, S}
  \begin{pmatrix}
    \e^{-s} I & 0 \\
    0 & \e^{s}I 
  \end{pmatrix}
  \bigg) \leq C^4\psi(s)^4
\end{equation}
for all  $S \subset \{1, \cdots, n\}$, and $s \geq 1$. 

\begin{lemma}
  \label{lemma:thetajbound}
  Suppose $\psi$ satisfies
  \begin{equation}
    \label{eq:psisum}
    \sum_{k \geq 0} \psi(k)^{-(2n + 2)} \leq C_\psi
  \end{equation}
  for some $C_\psi > 0$.
  Then 
  \begin{equation}
    \label{eq:badgset}
    \mu \left( \Gamma \backslash G - \mathcal{G}_{\bm{j}}(\psi, C) \right) \ll C_\psi C^{-( 2n + 2)}
  \end{equation}
  with the implied constant depending only on $n$. 
\end{lemma}

\begin{proof}
  Suppose $\Gamma g \not\in \mathcal{G}_{\bm{j}}(\psi, C)$, so
  \begin{equation}
    \label{eq:Dinequality1}
    D\bigg( \Gamma g g_{\bm{j}, S}
    \begin{pmatrix}
      \e^{-s} I & 0 \\
      0 & \e^s I
    \end{pmatrix}
    \bigg) \geq C^4 \psi(s)^4
  \end{equation}
  for some $S \subset \{1, \dots, n\}$ and $s \geq 1$.
  Applying lemma \ref{lemma:heightcontinuity} and the fact that $\psi$ is increasing, we find that there is an integer $k \geq 0$ such that
  \begin{equation}
    \label{eq:Dinequality2}
    D\bigg( \Gamma g g_{\bm{j}, S}
    \begin{pmatrix}
      \e^{-k} I & 0 \\
      0 & \e^k I
    \end{pmatrix}
    \bigg) \gg C^4  \psi(k)^4.
  \end{equation}  
  Applying lemma \ref{lemma:volumecalculation} together with the fact that right multiplication is volume preserving, we have that the volume of the set of $\Gamma g \in \Gamma \backslash G$ satisfying (\ref{eq:Dinequality2}) for a particular $S$ and $k$ is
  \begin{equation}
    \label{eq:volumebound1}
    \ll C^{- (2n + 2)} \psi(k)^{-(2n + 2)}.
  \end{equation}
  We obtain the required estimate by bounding the volume of the union of these sets over $S$ and $k$ by the sum of (\ref{eq:volumebound1}) over the relevant ranges.
\end{proof}

We now have all the ingredients for the proof of theorem \ref{theorem:thetasumbounds}. 

\begin{proof}[Proof of theorem \ref{theorem:thetasumbounds}]
  From (\ref{eq:chidecomp1}) we express $\theta_{\mathcal{B}}(M, X, \bm{x}, \bm{y})$ as 
  \begin{equation}
    \label{eq:Sdecomp}
    \sum_{S \subset \{1, \dots, n\}} \sum_{\bm{j} \geq 0} f_n \left( \frac{1}{M}(\bm{m} + \bm{x} + M\bm{x}_S B) B^{-1} E_S A_{\bm{j}}\right) \\
    \e\left( \frac{1}{2}( \bm{m} + \bm{x}) X \transpose{\!(\bm{m} + \bm{x})} + \bm{m} \transpose{\!\bm{y}} \right).
  \end{equation}
  Using (\ref{eq:chirepdecomp}), (\ref{eq:reprearrange}) we break the inner sum of (\ref{eq:Sdecomp}) as 
  \begin{multline}
    \label{eq:jsmalllarge}
    M^{\frac{n}{2}} (\det B)^{\frac12} \sum_{\substack{\bm{j} \geq 0 \\ 2^{j_i} b_{j_i}^{-1} \leq  M}}  2^{-\frac{1}{2}( j_1 + \cdots + j_n)} \Theta_{f_n}\left( h h_S^{g_{MB,X}^{-1}},  g_{MB,X} g_{\bm{j}, S} \right) \\
    +  \sum_{\substack{ \bm{j} \geq 0 \\ \max_i 2^{j_i} b_{j_i}^{-1} > M }} f_n \left( \frac{1}{M}(\bm{m} + \bm{x}+ M\bm{x}_S B) B^{-1} E_S A_{\bm{j}}\right) \e\left( \frac{1}{2}( \bm{m} + \bm{x}) X \transpose{\!(\bm{m} + \bm{x})} + \bm{m} \transpose{\!\bm{y}} \right),
  \end{multline}
  where $h = (\bm{x}, \bm{y}, 0)$ and
  \begin{equation}
    \label{eq:gXdef1}
    g_{B,X} =
    \begin{pmatrix}
      I & X \\
      0 & I
    \end{pmatrix}
    \begin{pmatrix}
      B^{-1} & 0 \\
      0 & B
    \end{pmatrix}
    .
  \end{equation}
  
  We first consider the second line of (\ref{eq:jsmalllarge}).
  Suppose that $L \subset \{1, \dots, n\}$ is not empty and that $2^{j_l} > b_{j_l} M$ for all $l \in L$.
  Then the compact support of $f_1$ implies that the sum over $\bm{m}^{(L)}$, the vector of entries of $\bm{m}$ with index in $L$, has a bounded number of terms.
  We write
  \begin{multline}
    \label{eq:miform}
    (\bm{m} + \bm{x}) X \transpose{\!( \bm{m} + \bm{x})} = (\bm{m}^{(L)} + \bm{x}^{(L)}) X^{(L,L)} \transpose{\!(\bm{m}^{(L)} + \bm{x}^{(L)})} \\
    + 2 (\bm{m}^{(L)} + \bm{x}^{(L)}) X^{(L,L')} \transpose{\!(\bm{m}^{(L')} + \bm{x}^{(L')})} +  (\bm{m}^{(L')} + \bm{x}^{(L')}) X^{(L',L')} \transpose{\!(\bm{m}^{(L')} + \bm{x}^{(L')})},
  \end{multline}
  where $L'$ is the complement of $L$, and $X^{(L_1,L_2)}$ is the matrix of entries of $X$ with row and column indices in $L_1$ and $L_2$ respectively.
  We have that $ f_n\left( \frac{1}{M} (\bm{m} + \bm{x} + M \bm{x}_S B) B^{-1} E_S A_{\bm{j}} \right) $ factors as
  \begin{multline}
    \label{eq:fLfLprime}
    f_{\# L} \left( \frac{1}{M} ( \bm{m}^{(L)} + \bm{x}^{(L)} + M\bm{x}_S^{(L)} ) (B^{(L,L)})^{-1} E_S^{(L,L)} A_{\bm{j}}^{(L,L)} \right) \\
    \times f_{\# L'} \left( \frac{1}{M}( \bm{m}^{(L')} + \bm{x}^{(L')} + M\bm{x}_S^{(L')} ) (B^{(L',L')})^{-1} E_S^{(L',L')} A_{\bm{j}}^{(L',L')} \right),
  \end{multline}
  and so, by inclusion-exclusion and the boundedness of $f_{\#L}$, the terms $\bm{j}$ of (\ref{eq:Sdecomp}) with $\bm{j}_l > b_{j_i} M$ for some $i$ is at most a constant (depending only on $n$) times
  \begin{equation}
    \label{eq:jlarge}
    \sum_{\substack{L \subset \{1, \dots, n\} \\ L \neq \emptyset}} \sum_{S \subset L} \sum_{\bm{m}^{(L)}} \big|\theta_{\mathcal{B}^{(L')}}(M, X^{L',L'}, \bm{x}^{(L')}, \bm{y}^{(L')} + (\bm{m}^{(L)} + \bm{x}^{(L)}) X^{(L,L')})\big|,
  \end{equation}
  where the sum over $\bm{m}^{(L)}$ has a bounded number of terms, $\mathcal{B}^{(L')}$ is the edge of $\mathcal{B}$ associatated to $L'$, and we have used the decomposition (\ref{eq:chidecomp1}) to express $\theta_{\mathcal{B}^{(L')}}(M, X^{L',L'}, \bm{x}^{(L')}, \bm{y}^{(L')} + (\bm{m}^{(L)} + \bm{x}^{(L)}) X^{(L,L')})$ as
  \begin{multline}
    \label{eq:reversedecomp}
    \sum_{S' \subset L'} \sum_{\bm{j}_{L'}} \sum_{\bm{m}_{L'}} f_{\# L'} \left( \frac{1}{M}( \bm{m}^{(L')} + \bm{x}^{(L')} + M\bm{x}_S^{(L')} ) (B^{(L',L')})^{-1} E_S^{(L',L')} A_{\bm{j}}^{(L',L')} \right) \\
    \times \e\left( \tfrac{1}{2}(\bm{m}^{(L')} + \bm{x}^{(L')}) X^{(L',L')} \transpose{\!(\bm{m}^{(L')} + \bm{x}^{(L')})} + \bm{m}^{(L')} \transpose{( \bm{y}^{(L')} + (\bm{m}^{(L)} + \bm{x}^{(L)}) X^{(L,L')})} \right).
  \end{multline}

  When $L = \{1, \dots, n\}$ or $n =1$, the corresponding part of (\ref{eq:jlarge}) is clearly bounded.
  Proceeding by induction on $n > 1$, for any other $L$, there are full measure subsets $\mathcal{X}^{(n - \#L)}$ such that if $X^{(L',L')} \in \mathcal{X}^{(n- \#L)}$, the corresponding part of (\ref{eq:jlarge}) is $\ll M^{ \frac{n - \#L}{2} + \epsilon}$ for any $\epsilon > 0$.
  It follows that (\ref{eq:jlarge}) is $\ll M^{\frac{n}{2}}$ assuming that $X$ is such that $X^{(L',L')} \in \mathcal{X}^{(n - \#L)}$ for all nonempty $L \subset \{1, \dots, n\}$.

  We now consider the part of (\ref{eq:jsmalllarge}) with $\bm{j}$ such that $2^{j_i} b_{j_i}^{-1} \leq  M$.
    We set $\mathcal{X}_{\bm{j}}(\psi, C)$ to be the set of $X \in \mathbb{Z}^{n\times n}_{\mathrm{sym}} \backslash \mathbb{R}^{n\times n}_{\mathrm{sym}}$ such that there exist $A \in \mathrm{GL}(n, \mathbb{R})$ and $T \in \mathbb{R}^{n\times n}_{\mathrm{sym}}$ satisfying $\sup_{B\in\mathcal{K}}|| A_{\bm{j}}^{-1} B A A_{\bm{j}}  - I || \leq \epsilon_n$, $||T|| \leq \epsilon_n$, and
  \begin{equation}
    \label{eq:ATcondition}
    \Gamma g = \Gamma
    \begin{pmatrix}
      I & X \\
      0 & I
    \end{pmatrix}
    \begin{pmatrix}
      A & 0 \\
      0 & \transpose{\!A}^{-1} 
    \end{pmatrix}
    \begin{pmatrix}
      I & 0 \\
      T & I
    \end{pmatrix}
    \in \mathcal{G}_{\bm{j}}(\psi, C).
  \end{equation}
  Here $\mathcal{K}$ is the compact subset in theorem \ref{theorem:thetasumbounds} identified with the compact subset of diagonal matrices $B$ in $\mathrm{GL}(n, \mathbb{R})$ in the obvious way.
  We then set
  \begin{multline}
    \label{eq:Xpsidef}
    \mathcal{X}(\psi) = \left( \mathbb{Z}^{n\times n}_{\mathrm{sym}}  + \bigcup_{C > 0} \bigcap_{\bm{j} \geq 0} \mathcal{X}_{\bm{j}}(\psi, C 2^{\frac{1}{4}( j_1 + \cdots + j_n)}) \right) \\
    \cap \bigcap_{L \subset \{1, \dots, n\} }  \{ X \in \mathbb{R}^{n\times n}_{\mathrm{sym}} : X^{(L', L')}  \in \mathcal{X}^{(n-\# L)} \} \subset \mathbb{R}^{n\times n}_{\mathrm{sym}}.
  \end{multline}

  We now verify that with $\psi$ satisfying the conditions of theorems \ref{theorem:thetasumboundssmooth}, \ref{theorem:thetasumbounds}, $\mathcal{X}(\psi)$ has full measure, noting (again by induction on $n$) that it is enough to show that
  \begin{equation}
    \label{eq:settobound}
    \bigcup_{C > 0} \bigcap_{\bm{j} \geq 0} \mathcal{X}_{\bm{j}}(\psi, C 2^{\frac{1}{4}( j_1 + \cdots + j_n)})
  \end{equation}
  has full measure in $\mathbb{Z}^{n\times n}_{\mathrm{sym}} \backslash \mathbb{R}^{n\times n}_{\mathrm{sym}}$. 
  First we suppose that the Lebesgue measure of the complement of $\mathcal{X}_{\bm{j}}(\psi, C)$ is greater than some $\epsilon > 0$.
  Then, using the expression (\ref{eq:XATHaar}) for the Haar measure on $G$, we have
  \begin{equation}
    \label{eq:complementmeasure}
    \mu\left( \Gamma \backslash G - \mathcal{G}_{\bm{j}}(\psi, C) \right) \gg \epsilon,
  \end{equation}
  with implied constant depending only on $n$ and $\mathcal{K}$.
  From lemma \ref{lemma:thetajbound} it follows that
  \begin{equation}
    \label{eq:complementmeasurebound}
    \mathrm{meas} \left( \mathbb{Z}^{n\times n}_{\mathrm{sym}} \backslash \mathbb{R}^{n\times n}_{\mathrm{sym}} - \mathcal{X}_{\bm{j},}(\psi, C) \right) \ll C_{\psi} C^{-2n -2},
  \end{equation}
  and we find that
  \begin{multline}
    \label{eq:complementmeasurebound1}
    \mathrm{meas}\left( \mathbb{Z}^{n\times n}_{\mathrm{sym}} \backslash \mathbb{R}^{n\times n}_{\mathrm{sym}} - \bigcup_{C > 0} \bigcap_{\bm{j} \geq 0} \mathcal{X}_{\bm{j}}(\psi, C 2^{\frac{1}{4}(j_1 + \cdots + j_n)}) \right) \\
    \ll  \lim_{C \to \infty}  \sum_{\bm{j} \geq 0} C_\psi C^{-2n-2} 2^{- \frac{n+1}{2}( j_1 + \cdots + j_n)} = 0
  \end{multline}
  as required.

  Now let us suppose that $X \in \mathcal{X}(\psi)$, so in particular the coset $\mathbb{Z}^{n\times n}_{\mathrm{sym}} + X$ is in $\mathcal{X}_{\bm{j}}(\psi, C 2^{\frac{1}{4}(j_1 + \cdots + j_n)})$ for some $C > 0$ (independent of $\bm{j}$) and all $\bm{j}\geq 0$.  
  We have from corollary \ref{corollary:upperbound} and the definition of the height function $D$ that 
  \begin{equation}
    \label{eq:thetabound}
    \ll M^{\frac{n}{2}} \sum_{S \subset \{1, \dots, n\}}  \sum_{\substack{\bm{j} \geq 0 \\ 2^{j_i} b_{j_i}^{-1} \leq  M}}  2^{-\frac{1}{2}( j_1 + \cdots + j_n)} D( \Gamma g_{MB, X} g_{\bm{j}, S})^{\frac{1}{4}}
  \end{equation}
  bounds the first line of (\ref{eq:jsmalllarge}). 
  Now for all $\bm{j} \geq 0$ there is a $g \in \mathcal{G}_{\bm{j}}(\psi, C 2^{\frac{1}{4}(j_1 + \cdots + j_n)})$ having the form
  \begin{equation}
    \label{eq:gXAT}
    g =
    \begin{pmatrix}
      I & X \\
      0 & I
    \end{pmatrix}
    \begin{pmatrix}
      A & 0 \\
      0 & \transpose{\!A}^{-1} 
    \end{pmatrix}
    \begin{pmatrix}
      I & 0 \\
      T & I
    \end{pmatrix}
  \end{equation}
  with $|| A_{\bm{j}}^{-1} B A A_{\bm{j}} - I || \leq \epsilon_n$ and $|| T || \leq \epsilon_n$.
  We have
  \begin{multline}
    \label{eq:grearrange}
    g g_{\bm{j}, S}
    \begin{pmatrix}
      \frac{1}{M} I  & 0 \\
      0 & M I
    \end{pmatrix}
    \\
    = g_{MB, X} g_{\bm{j}, S}
    \begin{pmatrix}
      E_S A_{\bm{j}}^{-1} B A A_{\bm{j}} E_S  & 0 \\
      0 &  E_S A_{\bm{j}} B^{-1} \transpose{\!A}^{-1} A_{\bm{j}}^{-1} E_S 
    \end{pmatrix}
    \begin{pmatrix}
      I & 0 \\
      \frac{1}{M^2} A_{\bm{j}} T A_{\bm{j}} & I 
    \end{pmatrix}
    ,
  \end{multline}
  and so lemma \ref{lemma:heightcontinuity2} implies
  \begin{equation}
    \label{eq:DgDgMX}
    D( g_{MB, X} g_{\bm{j}, S}) \asymp D\left(  g g_{\bm{j}, S}
      \begin{pmatrix}
        \frac{1}{M} I & 0 \\
        0 & M I 
      \end{pmatrix}
\right) \ll C 2^{\frac{1}{4}(j_1 + \cdots + j_n)} \psi(\log M)
  \end{equation}
  since $g \in \mathcal{G}_{\bm{j}}(\psi, C 2^{\frac{1}{4}(j_1 + \cdots + j_n)})$ and $2^{j_i} \leq M b_i$ gives
  \begin{equation}
    \label{eq:AjTinequality}
    \frac{1}{M^2} || A_{\bm{j}} T A_{\bm{j}}  || \leq || T || \leq \epsilon_n. 
  \end{equation}
  It follows that (\ref{eq:thetabound}) is bounded by
  \begin{equation}
    \label{eq:thetabound1}
    \ll C M^{\frac{n}{2}}\psi(\log M) \sum_{\bm{j} \geq 0} 2^{- \frac{1}{4}(j_1 + \cdots + j_n)}  \ll C M^{\frac{n}{2}} \psi( \log M),
  \end{equation}
  and theorem \ref{theorem:thetasumbounds} follows.
\end{proof}

\bibliographystyle{plain}
\bibliography{references.bib}

\end{document}